\documentclass[reqno,12pt]{amsart}
\usepackage{mathtools}
\usepackage{microtype}
\usepackage{enumerate}
\usepackage{csquotes}
\usepackage[normalem]{ulem}
\usepackage{mathrsfs}

\calclayout

\DeclareMathOperator{\curl}{curl}

\DeclareMathOperator{\Span}{span}
\DeclareMathOperator{\Ker}{Ker}
\DeclareMathOperator{\supp}{supp}
\DeclareMathOperator{\dist}{dist}

\DeclarePairedDelimiter{\IP}{\langle}{\rangle}
\DeclarePairedDelimiter{\abs}{\lvert}{\rvert}
\DeclarePairedDelimiter{\norm}{\lVert}{\rVert}

\newcommand{\C}{\mathbb C}

\newcommand{\Z}{\mathbb Z}
\newcommand{\R}{\mathbb R}
\newcommand{\N}{\mathbb N}

\newcommand{\Ab}{\mathbf A}
\newcommand{\dd}{\mathop{}\!{\mathrm{d}}}

\newcommand{\ii}{\mathrm i}
\newcommand{\ee}{\mathrm e}

\usepackage{xcolor}

\definecolor{blue(ryb)}{rgb}{0.01, 0.28, 1.0}
\definecolor{brandeisblue}{rgb}{0.0, 0.44, 1.0}
\definecolor{ceruleanblue}{rgb}{0.16, 0.32, 0.75}
\definecolor{cobalt}{rgb}{0.0, 0.28, 0.67}
\definecolor{coolblack}{rgb}{0.0, 0.18, 0.39}
\definecolor{darkblue}{rgb}{0.0, 0.0, 0.55}
\usepackage[hidelinks]{hyperref}
\hypersetup{
    colorlinks,
    citecolor=darkblue,
    filecolor=black,
    linkcolor=darkblue,
    urlcolor=black
}

\newtheorem{theorem}{Theorem}[section]

\newtheorem{proposition}[theorem]{Proposition}
\newtheorem{definition}[theorem]{Definition}
\newtheorem{corollary}[theorem]{Corollary}
\theoremstyle{remark}
\newtheorem{remark}[theorem]{Remark}

\newtheorem{assumption}[theorem]{Assumption}

\numberwithin{equation}{section}

\title[Flux effects on quantum tunneling]{Flux  and symmetry effects on quantum tunneling}

\author[B. Helffer]{Bernard Helffer}
\author[A. Kachmar]{Ayman Kachmar}
\author[M.P. Sundqvist]{Mikael Persson Sundqvist}
\address[B. Helffer]{Laboratoire de Math\'ematiques Jean Leray, CNRS,  Nantes Universit\'e, France.}
\email{bernard.helffer@univ-nantes.fr}
\address[A. Kachmar]{Department of Mathematics, Lebanese University, Nabatiyeh, Lebanon.}
\address[Current address: The Chinese University of Hong Kong (Shenzhen), School of Science and Engineering, 2001 Longxiang Blvd., Longgang District, Shenzhen, China.]{}
\email{ayman.kashmar@gmail.com}
\address[M.P. Sundqvist]{Department of Mathematics, Lund University, Sweden}
\email{mikael.persson\_sundqvist@math.lth.se}

\subjclass{81U26}
\keywords{Schrödinger operator,tunneling,magnetic field,eigenvalue splitting}

\begin{document}
\begin{abstract}
Motivated by the analysis of the tunneling effect for the magnetic Laplacian,   we introduce an abstract framework for the spectral reduction of a self-adjoint operator to a hermitian matrix. We illustrate this framework by three applications, firstly the electro-magnetic Laplacian with constant magnetic field and three equidistant potential wells, secondly a pure constant magnetic field and Neumann boundary condition in a smoothed triangle, and thirdly a magnetic step where the discontinuity line is a smoothed triangle. Flux effects are visible in the three aforementioned settings  through the occurrence of eigenvalue crossings. Moreover, in the electro-magnetic Laplacian setting with double well radial potential, we rule out an artificial condition on the distance of the wells and extend  the range of validity for the tunneling approximation recently established in~\cite{FSW, HK22}, thereby settling the problem of electro-magnetic tunneling under constant magnetic field and  a sum of translated radial electric potentials.
\end{abstract}
\maketitle

\section{Introduction}

\subsection{Tunneling and flux effects}

The tunneling  induced by symmetries is an interesting phenomenon in spectral theory featuring  an exponentially small splitting between the ground state and the next  excited state energies.  The magnetic flux has an effect on the eigenvalue multiplicity which can lead to oscillatory patterns in the spectrum: as the magnetic flux varies,  the eigenvalues may cross and split infinitely many times.  Hence it is interesting to look at the interaction between symmetry and flux effects.
We explore this question by investigating  examples of operators involving the magnetic Laplacian \((-\ii h\nabla-\Ab)^2\) perturbed in various ways by an electric potential,  a boundary condition or a magnetic field discontinuity.  We observe interesting flux effects,   manifested in endless eigenvalue crossings, when adding  symmetry assumptions on the electric potential,  the  boundary of the domain or the magnetic field discontinuity set.
  
A braid structure in the distribution of the low lying eigenvalues was predicted heuristically~\cite[Sec.~15.2.4]{FH-b} and confirmed numerically~\cite{BDMV} for the magnetic Laplacian on an equilateral triangle  with Neumann boundary condition and constant magnetic field.  We confirm this prediction by giving a proof for smoothed equilateral triangles and for other examples on the full plane (electro-magnetic Laplacian and magnetic steps).

Let us introduce a mathematical definition of (semi-classical) \emph{braid structure} of lowest eigenvalues.   Consider a family of   unbounded self-adjoint  operators \((T_h)_{h\in (0,1]}\)
on a Hilbert space \(H\).  Let us assume that,  for every \(h\in(0,1]\),  \(T_h\) is semi-bounded from below and denote by \(\lambda_1(T_h),\lambda_2(T_h),\cdots\) the  discrete eigenvalues
below the essential spectrum of \(T_h\),  counted with multiplicity.  In practical  examples,  the parameter \(h\) will be the semi-classical parameter, which tends to \(0\) and can result as the inverse of the magnetic field's intensity in problems involving  strong  magnetic fields.

\begin{definition}\label{def:BS}
The lowest eigenvalues of  \(T_h\),   \(\lambda_1(T_h)\) and \(\lambda_2(T_h)\),  are said to have a braid structure,  if 
\[
  \forall\,\epsilon>0,  ~\exists\,h_\epsilon,h'_\epsilon\in(0,\epsilon),\
  \begin{cases}
    \lambda_2(T_{h_\epsilon})-\lambda_1(T_{h_\epsilon})>0\\
    \lambda_2(T_{h'_\epsilon})-\lambda_1(T_{h'_\epsilon})=0
  \end{cases}.
\]
\end{definition}  

According to Definition 1.1, the  eigenvalues \(\lambda_1(T_h)\) and \(\lambda_2(T_h)\) exhibit infinitely many crossings and splittings as the parameter \(h\) varies in a right neighborhood of \(0\) (see Fig.~\ref{braid}). This phenomenon has  also been observed in non-simply connected domains when considering a magnetic Laplacian and assuming an Aharonov--Bohm magnetic potential. Furthermore, it has been proven in~\cite{HHHO} that these crossings and splittings occur in response to variations in the magnetic flux.

\begin{figure}[htb]
  \includegraphics[page=5]{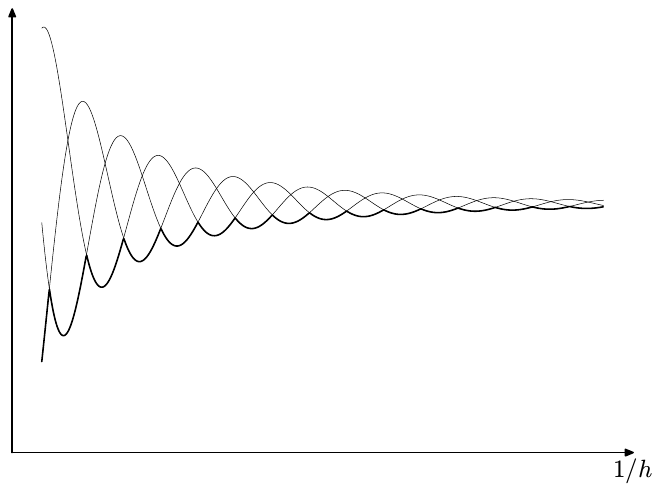}
  \caption{A schematic figure of eigenvalues with a braid structure, occurring in the presence of trilateral symmetry.   The ground state energy has multiplicity \(2\) infinitely many times.  Observe also that the energy of the second state may have multiplicity \(2\) while the ground state energy is a simple eigenvalue.}\label{braid}
\end{figure}

For the semi-classical magnetic Laplacian on a simply connected domain with Neumann boundary conditions,  the spectrum  is  related with the spectral properties of  an operator which is defined on the boundary.  Hence we actually   work on another  non-simply connected domain (i.e. the boundary)  and therefore flux effects are expected to exhibit crossings and splittings of eigenvalues. However, this is not the case when the boundary curvature has for example  a unique non-degenerate  maximum.  In this   case the eigenvalues split in the semi-classical limit~\cite{FH06}.  It is when non-degenerate minima are exchanged in the presence of symmetries (like in the case of an ellipse or a smoothed equilateral triangle),  that eigenvalue crossings are expected to occur along with tunneling effects~\cite{BHR,FH-b}. We will also prove  such a behavior for the electro-magnetic Laplacian with `potential' wells located on the vertices of an equilateral triangle,  which to our knowledge is quite novel.

\subsection{Electro-magnetic tunneling}

The analysis in this paper yields new results on the electro-magnetic Laplacian on \(\R^2\),
\[
  \mathcal L_{h,b}
  =
  (-\ii h\nabla-b\Ab)^2+V,
\]
where \(b,h\) are positive parameters,  \(\Ab=\frac12(-x_2,x_1)\) is the vector field generating the unit uniform magnetic field, \(\curl\Ab=1\), and \(V\) is a  smooth function. 

What we call the wells are the points where \(V\) attains its minimum. 
The pure electric case where $b=0$ was settled for any number of wells $n$ in~\cite{HeSj2}. We would like to address the case where \(b>0\) and \(n\geq 2\).  For $n=2$, this problem was considered in~\cite{HeSjPise} and revisited recently in~\cite{FSW, HK22}. The article  
\cite{HeSjPise} follows a perturbative approach (i.e.  considers the case $b$ relatively small) and assumes the  analyticity of the electric potential \(V\),  while the results in~\cite{FSW, HK22} hold for any $b>0$ but under the assumption that \(V\) is non-positive and  defined as a superposition of radially symmetric compactly supported functions.

\subsubsection{Double wells}

Suppose that the electric potential \(V\) is as follows
\begin{equation}\label{eq:int-V-n=2}
V(x)=v_0(|x-z_1|)+v_0(|x-z_2|)
\end{equation}
where \(v_0\in C^\infty(\overline{\R_+})\) vanishes on \([a,+\infty)\),  negative-valued on \([0,a)\) and has a unique  and non-degenerate minimum at \(0\). The wells of \(V\) are then \(z_1\) and \(z_2\).  
We prove the following theorem that, in particular, rules out eigenvalue crossings for double wells.

\begin{theorem}\label{thm:HKS1}
Assuming \(b>0\) is fixed and   \(V\) is given as in~\eqref{eq:int-V-n=2} with \(z_1=0\),  \(z_2=(L,0)\) and \(L>4a\),  then there exists a positive constant \(\mathscr E_{b,L}(v_0)\) such that
  \begin{equation}\label{eq:HKS1}
    h \ln\big(\lambda_2(\mathcal L_{h,b})-\lambda_1(\mathcal L_{h,b})\big) \underset{h\searrow0}{\sim} -  \mathscr E_{b,L}(v_0).
  \end{equation}
Moreover
  \[
    \mathscr E_{b,L}(v_0)\underset{b\searrow0}{\sim}2\int_0^{L/2}\sqrt{v_0(\rho)-v_0^{\min}}\dd \rho,
  \]
where \(v_0^{\min}=\min_{r\geq 0}v_0(r)\).
\end{theorem}

\begin{remark}\label{rem:HKS1}~
\begin{enumerate}[i)]
  \item
  The asymptotics in~\eqref{eq:HKS1} was obtained earlier in~\cite{HK22} for \(b=1\) but under the assumption that
  \begin{equation}\label{assFSW}
    L > 4\Bigl(\sqrt{|v_0^{\mathrm{min}}|}+a\Bigr).
  \end{equation}
  \item 
  The use of~\eqref{assFSW} in~\cite{HK22} was technical.   In fact,  assuming~\eqref{assFSW},  it is proved in~\cite{FSW} that 
  \begin{equation}\label{eq:ho-co}
    \lambda_2(\mathcal L_{h,1})-\lambda_1(\mathcal L_{h,1})\underset{h\searrow0}{\sim}  \mathfrak c_h( v_0,L)
  \end{equation}
  where \(\mathfrak c_h(v_0,L)\) is the hopping coefficient that will be introduced in~\eqref{eq:hopping} later on.   The accurate approximation of \(\ln \mathfrak c_h(v_0,L)\) was then carried out in~\cite{HK22}.  
  \item 
  For \(b\not=1\),  by a change of semi-classical parameter,  the condition in~\eqref{assFSW} reads as follows
  \[
    L>4\Bigl(b^{-1}\sqrt{|v_0^{\mathrm{min}}|} + a\Bigr).
  \] 
  Clearly,  this is a very strong condition on \(L\) which in particular  prevents us  from  considering the limit \(b\searrow 0\).  The novelty in Theorem~\ref{thm:HKS1} is in improving the previous  condition which could appear as artificial.  However,  it is still an open question whether~\eqref{eq:HKS1} holds for \(2a<L<4a\).
  \item 
  The dependence on $b$ in the expression of \(\mathscr E_{b,L}(v_0)\) can  actually be made more explicit. We have indeed
  \[
    \mathscr E_{b,L}(v_0)=  b\, S(b^{-2}v_0,L)\,,
  \]
  where \(S(\cdot,L)\) will be introduced in~\eqref{eq:def-S(v0)**} later on.

  The leading term  of \(\mathscr E_{b,L}(v_0)\) in the limit \(b\searrow 0\) was calculated in~\cite[Prop.~6.7]{HK22},  and it is consistent with the existing results~\cite{HeSj1, Si2} without a magnetic field, \(b=0\), thereby showing a sort of continuity of the tunneling estimate with respect to the magnetic field's strength.
\end{enumerate}
\end{remark}
 
\subsubsection{Three wells and eigenvalue crossings}

Suppose now that the electric potential \(V\) is as follows
\begin{equation}\label{eq:int-V-n=3}
  V(x)=v_0(|x-z_1|)+v_0(|x-z_2|)+v_0(|x-z_3|)
\end{equation}
where \(v_0\) is the same function as in~\eqref{eq:int-V-n=3} and that the wells \(z_1,z_2,z_3\) are located on the vertices of an equilateral triangle with side length \(L\).  We prove then the existence of a braid structure in the sense of Definition~\ref{def:BS}.

\begin{theorem}\label{thm:HKS2}
Assuming \(b>0\) is fixed and   \(V\) is given as in~\eqref{eq:int-V-n=3} with 
\[
  |z_1-z_2|=|z_2-z_3|=|z_3-z_1|=L>4a,
\] 
then the lowest egigenvalues of \(\mathcal L_{h,b}\) has a braid structure.  Moreover,
\[
  \limsup_{h\searrow 0} \Bigl(h \ln\big(\lambda_2(\mathcal L_{h,b})-\lambda_1(\mathcal L_{h,b})\big)\Bigr) = - \mathscr E_{b,L}(v_0)
\]
with \(\mathscr E_{b,L}(v_0)\) the same negative quantity as in Theorem~\ref{thm:HKS1}.
\end{theorem}

Not only Theorem~\ref{thm:HKS2} establishes the existence of infinitely many eigenvalue crossings and splittings,  but it also establishes  an accurate estimate for the magnetic tunneling induced by three symmetric potential wells,  thereby extending the recent results of~\cite{FSW, HK22} on double wells.  

\subsection{Geometrically induced braid structure}

We present here results on the pure magnetic Laplacian where the eigenvalue crossings are induced by a combination of the geometry and the flux in the semi-classical limit. We shall describe the results when \(\Omega\) is a smoothed triangle  (see Fig~\ref{fig1-intro}).  That is,  \(\Omega\) is a simply connected domain,  invariant under  rotation  by \(2\pi/3\),  with three points of maximum curvature that are equidistant with respect to the arc-length distance on the boundary.  

\begin{figure}[h]
  \includegraphics[page=3]{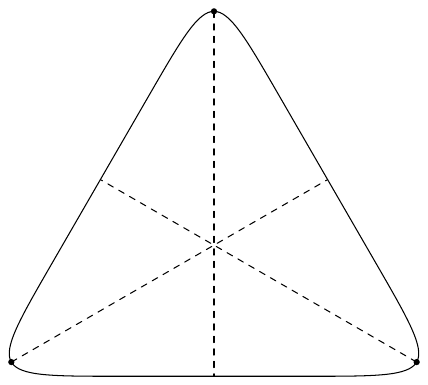}
  \caption{Illustration of the domain \(\Omega\): a smoothed triangle invariant under the rotation by \(2\pi/3\).}
  \label{fig1-intro}
\end{figure}

\subsubsection{The magnetic Neumann Laplacian under constant magnetic field}

In the Hilbert space  \(L^2(\Omega)\),  we consider the magnetic Neumann Laplacian \(\mathscr L_h^N=(-\ii h\nabla-\Ab)^2\) on \(\Omega\),  with uniform  magnetic field \(\curl\Ab=1\) and  (magnetic) Neumann condition
\[
  \nu\cdot (-\ii h\nabla-\Ab)u|_{\Gamma}=0.
\]

\begin{theorem}\label{thm:HKS3}
Let the domain \(\Omega\) be a smoothed triangle,  invariant under the rotation by \(2\pi/3\)  (see Fig.~\ref{fig1-intro}). The lowest eigenvalues of the operator \(\mathscr L_h^N\) have a braid structure.
\end{theorem}

The presence of the Neumann boundary condition plays a vital role in the preceding theorem. This condition is responsible for the semi-classical localization of the bound states near the boundary of \(\Omega\) and this  is this localization that gives rise to the observed flux effects.

\subsubsection{The Landau Hamiltonian under a magnetic step}

We consider here the Landau Hamiltonian  \(\mathscr L_h^B=(-\ii h\nabla-\Ab)^2\) on \(\R^2\)
with the discontinuous magnetic field
\[
  B
  =\curl\Ab
  =
  \begin{cases}
    1         &\text{ on }\Omega,\\
    \vartheta\in(-1,0)&\text{ on }\R^2\setminus\overline\Omega.
  \end{cases}
\]
Such magnetic fields have been called `magnetic steps' in the literature, and the semi-classical limit for the operator \(\mathcal L_h^B\) has been  studied recently in~\cite{AHK, FHK} (and in~\cite{FHKR} for \(\vartheta =-1\)). The bound states of the system become increasingly concentrated along the discontinuity of $B$ in the semi-classical limit. Consequently, we expect the emergence of flux effects, and they are indeed realized when \(\Omega\) is a smoothed triangle with  $3$-fold  symmetry.

\begin{theorem}\label{thm:HKS4}
Assume that \(\Omega\) is a smoothed  triangle,  invariant under the rotation by \(2\pi/3\) (see Fig.~\ref{fig1-intro}). The lowest eigenvalues of the operator  \(\mathscr L_h^B\) have in the semi-classical limit a braid structure.
\end{theorem}

We therefore have an example where the eigenvalues of the  Landau Hamiltonian on the full plane cross and split infinitely many times.  This is the consequence of having a sign changing magnetic field with a discontinuity along a simple smooth curve.  

Recently, an estimate for the  tunneling induced by  a smooth magnetic field that vanishes non-degenerately along a smooth curve has been established in~\cite{Al}. It  seems  natural to expect the existence of a braid structure in that setting too when adding symmetry assumptions.

\subsection{Organization}

The paper is organized as follows.  Since we work with various symmetry configurations, Section~\ref{sec.IM} is devoted to an abstract spectral reduction to a hermitian matrix (so often called the interaction matrix in the literature on tunneling effects~\cite{HeSj1, He88}).  This provides us with a robust methodology when analyzing tunneling effects in various settings. Loosely speaking,  all we need is the construction of adequate quasi-modes.

In Section~\ref{sec.HK22} we discuss the electro-magnetic Laplacian.  Our investigation has two ingredients,  the first is to appply the abstract methodology in Section~\ref{sec.IM},  and the second is to control  the errors produced by the interaction terms; the later task is achieved by using the analysis  in the recent work~\cite{HK22}.  Section~\ref{sec.HK22} concludes with the proofs of Theorems~\ref{thm:HKS1} and \ref{thm:HKS2}.

In Section~\ref{sec.Neumann},  we prove Theorem~\ref{thm:HKS3} by applying the results of Section~\ref{sec.IM}. The control of the error terms and the computation of the interaction term were done in~\cite{BHR}.

Finally,  in Section~\ref{sec.edge},  we prove Theorem~\ref{thm:HKS4} by applying the constructions in Section~\ref{sec.IM}.  We will be more succinct here since the analysis is similar to Section~\ref{sec.Neumann}.  The control of the error terms and the asymptotics of the  interaction terms were done  in~\cite{FHK} and are actually rather close to those in~\cite{BHR}.

\section{Abstract framework for a spectral reduction to a  hermitian matrix }\label{sec.IM}

In this section we consider a family of operators dependent on a positive semi-classical parameter \(h\ll 1\).  Assuming the existence of certain quasi-modes (see Assumption~\ref{ass1} below),  we can approximate the eigenvalues of the operator with those of the matrix of its restriction on a specific basis (Proposition~\ref{prop:int-matrix} below).  Assuming an additional symmetry hypothesis (Assumption~\ref{ass:sym} below) we can relabel the eigenvalues and spot their possible crossings and splittings as the semi-classical parameter approaches \(0\) (Eqs.~\eqref{eq:formn=2},  \eqref{eq:formn=3} and Paragraphs~\ref{sec:2-wells}, \ref{sec:3-wells}).  

\subsection{Preliminaries}

Consider a Hilbert space \(H\) endowed with an inner product \(\langle \cdot,\cdot\rangle\) and a family of self-adjoint unbounded operators \(T_h\colon D_h\to H\), \(h\in(0,1]\). Assume furthermore that, for every \(h\in(0,1]\),  \(T_h\) is semi-bounded from below and has a sequence of discrete eigenvalues
\[
  \lambda_1(h)
  \leq 
  \lambda_2(h)
  \leq
  \lambda_3(h) 
  \leq
  \ldots
  <
  \Sigma_h \coloneqq  \inf \sigma_{\mathrm{ess}}(T_h) \in \R \cup \{+\infty\},
\]
counted with multiplicity. By the min-max principle the eigenvalues can be represented as
\[
  \lambda_n(h)
  =
  \inf _ {\substack{M\subset D_h\\ \dim(M)=n}}
  \left(
    \max_{\substack{\phi\in M\\ \|\phi\|=1}}
    \langle T_h\phi,\phi\rangle
  \right).
\]
We will work under  additional assumptions on the operators \((T_h)_{h\in(0,1]}\).

\begin{assumption}[\(n\) wells]\label{ass1}
  Let \(n\geq 2\) be an integer. There exist positive constants $ \mathfrak S_1,\mathfrak S_2,\mathfrak S_3, c,p,q$ and $h_0 \in (0,1]$ such that $p<q$ and for all \(h\in(0,h_0]\), there exists a subspace \(E_h= \Span (u_{h,1},\ldots ,u_{h,n})\subset D_h\) such that:
  \begin{enumerate}
  \item 
  $\max_{1\leq i \leq n}  \| T_h u_{h,i} \| =\mathcal O(\ee^{-\mathfrak S_1/h}).$
  \item 
  \(
    \langle u_{h,i},u_{h,j} \rangle =
    \begin{cases}
      1+\mathcal O(\ee^{-\mathfrak S_2/h}) & i = j, \\
      \mathcal O(\ee^{-\mathfrak S_3/h})   & i \neq j.
    \end{cases}
  \)
  \item
  \( \lambda_1(h)\geq -c h^{q} \).
  \item  
  \( \lambda_{n+1}(h) \geq c\,h^{p} \).
\end{enumerate}
\end{assumption}

It results from (2) above that \(\dim(E_h)=n\) for \(h\) small enough. As a consequence of Assumption~\ref{ass1}, we now prove that the operator \(T_h\) has precisely \(n\) eigenvalues that are \emph{exponentially small} in \(h\),  and  there is a gap to \(\lambda_{n+1}(h)\), since it is at least of polynomial size in \(h\).

\begin{proposition}\label{prop:ev-bounds}
  Under Assumption~\ref{ass1},  there exist positive constants \(C,h_1\) such that,  for $h\in (0,h_1]$,
  \begin{equation}\label{eq:lambdan} 
     \lambda_n(h)\leq C e^{-\mathfrak{S}_1/h}.
  \end{equation}
  In particular \(\lambda_n(h)<\lambda_{n+1}(h)\) for \(h\) sufficiently small.
\end{proposition}

\begin{proof}
Since \( \dim(E_h)=n \), we can use the min-max principle. Let \( \phi = \sum_{j}\alpha_ju_{h,j} \) be in \(E_h\). Then, the triangle inequality and Cauchy--Schwarz inequality, together with (1) and (2) from Assumption~\ref{ass1}, provide the existence of constants \(A_1\) and \(A_2\), such that if \(h\) is small enough,
\[
  \IP{T_h\phi,\phi}
  \leq 
  \sum_{j,k} \abs{\alpha_j}\cdot \abs{\alpha_k} \cdot \norm{ T_h u_{h,j} } \cdot \norm{u_{h,k}}
  \leq
  A_1 \ee^{-\mathfrak{S}_1/h}(1+A_2\ee^{-\mathfrak{S}_2/h}).
\]
On the other hand, we can use (2) of Assumption~\ref{ass1} to bound the norm of \(\phi\) from below. We get indeed constants \(A_3, A_4\) and \(A_5\) such that, if \(h\) is small enough,
\[
  \norm{\phi}^2 \geq A_3(1-A_4\ee^{-\mathfrak{S}_2/h} - A_5\ee^{-\mathfrak{S}_3/h}).
\]
This gives the bound in~\eqref{eq:lambdan}. Combining this bound with (4) in Assumption~\ref{ass1} we conclude that that \(\lambda_{n+1}(h) > \lambda_n(h)\) if \(h\) is sufficiently small.
\end{proof}

We want to link the quasi mode constructions \(\{u_{h,j}\}\) to the low-lying eigenvalues of \(T_h\). To do this, we want to show that the symmetric matrix \(\mathsf{U}_h=(\mathsf{u}_{j,k})\),
\begin{equation}\label{eq:def-u-jk}
  \mathsf{u}_{j,k} = \IP{T_h u_{h,j},u_{h,k}},
\end{equation}
does not differ much (component-wise) from the matrix \(\mathsf{W}_h\) that  will be constructed as the restriction of \(T_h\) to the eigenspace
\begin{equation}\label{eq:def-Fh}
 F_h\coloneqq \bigoplus_{j=1}^n \Ker (T_h-\lambda_j(h)),
\end{equation}
written in an orthonormal basis. We will do the approximation in two steps. We first consider the projected functions 
\[
  v_{h,j} = \Pi_{F_h}u_{h,j}
\]
and show that the norms \(\norm{v_{h,j} - u_{h,j}}\) are small. Since the span of \(\{u_{h,j}\}\) is \(n\)-dimensional by Assumption~\ref{ass1} (2), it will follow that the \(\{v_{h,j}\}\) are linearly independent, and thus constitute a basis for \(F_h\). 

\begin{proposition}\label{prop:dim-F}
  If Assumption~\ref{ass1} holds for $E_h$, then for \(h>0\) sufficiently small, we have \(\dim (F_h)=n\), and the vectors
  \[
    v_{h,i}=\Pi_{F_h}u_{h,i}\quad (i\in\{1,\cdots,n\}),
  \]
  form a basis of \(F_h\). Moreover they satisfy 
  \begin{equation}\label{eq:2.2}
    \max_{1\leq i\leq n}
    \norm{v_{h,i}-u_{h,i}}
    =
    \mathcal O(h^{-p}\ee^{-\mathfrak S_1/h}).
  \end{equation}
\end{proposition}

\begin{proof}
  Since we count multiplicities, we know in general that \(\dim (F_h)\geq n\). However, by (4) in Assumption~\ref{ass1}, we get from Proposition~\ref{prop:ev-bounds} that \(\dim(F_h)=n\).

  With \(v_{h,i}=\Pi_{F_h} u_{h,i}\) we note that \(u_{h,i} - v_{h,i} \in H \ominus F_h\). Since \(T_h\), restricted to \(H \ominus F_h\) is bounded below by \(ch^p\) by Assumption~\ref{ass1} (4), we find that
  \[
    \norm{T(u_{h,i} - v_{h,i})} \geq c\, h^p\norm{u_{h,i}-v_{h,i}}.
  \] 
  On the other hand, According to Assumption~\ref{ass1} (1) and (2), and Proposition \ref{prop:ev-bounds},
  \[
    \norm{T(u_{h,i} - v_{h,i})}
    \leq
    \norm{Tu_{h,i}} + \norm{T v_{h,i}}
    \leq
    C\ee^{-\mathfrak{S}_1/h}.
  \]
  Combining these inequalities we get~\eqref{eq:2.2}. From this and Assumption~\ref{ass1} (2), we find that \(\{v_{h,1},\ldots,v_{h,n}\}\) are linearly independent, and hence a basis for \(F_h\).
\end{proof}

\subsection{Reduction to a matrix through a suitable orthonormal basis}
The aim in this subsection is to find an orthonormal basis for $F_h$ such that the matrix of the restriction of $T_h$ in this basis can be well approximated. Later in the applications to multiple wells problems this matrix will be according to the previous literature called the interaction matrix.

The basis \(\{v_{h,j}\}\) of \(F_h\) that we just constructed will, in general, not be orthogonal. We construct, by a symmetry-preserving Gram--Schmidt procedure an orthonormal basis \(\{w_{h,j}\}\). The matrix \(\mathsf{W}_h\) will be the matrix of \(T_h\) restricted to \(F_h\), written in this new basis \(\{w_{h,j}\}\). 
 
Let us denote by \(\mathsf G_h=(\mathsf g_{ij}(h))_{1\leq i,j\leq n}\) the Gram matrix of the basis \(\{v_{h,1},\ldots,v_{h,n}\}\) of \(F_h\), where
\begin{equation}\label{eq:def-Gij}
  \mathsf g_{ij}=\langle v_{h,i},v_{h,j}\rangle.
\end{equation}
Since the  \(\{v_{h,j}\}\) are linearly independent, the Gram matrix becomes positive definite, so \(\mathsf G_h^{-1/2}\) is well defined and positive definite. We obtain an orthonormal basis \(\mathcal V_h=\{w_{h,1},\ldots,w_{h,n}\}\)  of \(F_h\) as follows\footnote{We could  have worked in the basis obtained by the standard Gram--Schmidt process but the downside is that the Gram--Schmidt process does not respect the symmetry invariance that we will impose later.}
\begin{equation}\label{eq:def-fij}
  \begin{pmatrix}
    w_{h,1}\\ \vdots\\w_{h,n}
  \end{pmatrix}
  = 
  \mathsf G_h^{-1/2}
  \begin{pmatrix}
    v_{h,1}\\ \vdots\\v_{h,n}
  \end{pmatrix}.
\end{equation}

We consider the restriction of \(T_h\) to the space \(F_h\) and  denote by \(\mathsf W_h=(\mathsf w_{ij})_{1\leq i,j\leq n}\) its  matrix  in the basis \(\mathcal V_h\), so \(\mathsf{w}_{ij}=\IP{T_h w_{h,i},w_{h,j}}\). The matrix \(\mathsf W_h\) is hermitian, with eigenvalues \(\{\lambda_1(h),\ldots,\lambda_n(h)\}\).  The next proposition controls how   \(\mathsf W_h\) is approximated by the matrix \(\mathsf U_h\) defined by~\eqref{eq:def-u-jk}.  

\begin{proposition}\label{prop:int-matrix}
  Let 
  \begin{equation}\label{eq:def-Lambda-h}
    \left.
    \begin{aligned}
      \Lambda_h &\coloneqq \| T_h|_{F_h}\| =\max_{1\leq i\leq n}|\lambda_i(h)|,\\
      C_h &= h^{-p}\ee^{-\mathfrak S_1/h}+\ee^{-\min(\mathfrak S_2,\mathfrak S_3)/h},\\
      \varepsilon_h &=(\Lambda_h+C_h)C_h . 
    \end{aligned}
    \right\}
  \end{equation}
  Then the matrix \(\mathsf R_h = (\mathsf{r}_{ij})\),
  \begin{equation}\label{eq:inter-a}
    \mathsf R_h
    \coloneqq 
    \mathsf{W}_h - \mathsf U_h,
  \end{equation}
  is symmetric, and satisfies
  \begin{equation}\label{eq:inter-c} 
    \mathsf r_{ij}
    = 
      \mathcal O(\varepsilon_h)\,.
  \end{equation}
\end{proposition}

\begin{proof}~\\
  {\bf Step 1.} Proposition~\ref{prop:dim-F} says that 
  \[
    \langle v_{h,i},v_{h,j}\rangle=\langle u_{h,i},u_{h,j}\rangle+\mathcal O(h^{-p}\ee^{-\mathfrak S_1/h}).
  \]
  With \(\mathsf I\) denoting the \(n\times n\) identity matrix, we get from (2) in Assumption~\ref{ass1},
  \[
    \mathsf G_h=\mathsf I+\mathcal O(C_h)
  \]
  so that
  \[
    \mathsf G_h^{-1/2}=\mathsf I+\mathcal O(C_h).
  \]
  It follows then that
  \begin{equation}\label{eq:f=v}
    \left.
    \begin{aligned}
    \|w_{h,i}-v_{h,i}\|&=\mathcal O(C_h),\\
    \|T_h w_{h,i}-T_h v_{h,i}\|&=\mathcal O\big(\Lambda_hC_h\big)
    \end{aligned}
    \right\}
  \end{equation}
  where in the second estimate in~\eqref{eq:f=v}, we simply combine the first estimate and the definition of $\Lambda_h$.

  {\bf Step 2.} We may write
  \begin{multline*}
    \mathsf w_{ij}
    =
    \langle T_{h}w_{h,i},w_{h,j}\rangle
    =
    \langle T_hv_{h,i},v_{h,j}\rangle
    +
    \langle T_hv_{h,i},w_{h,j}-v_{h,j}\rangle 
    +    \langle T_h (w_{h,i}-v_{h,i}),w_{h,j}\rangle.
  \end{multline*}
 Using this identity,~\eqref{eq:f=v} and Proposition~\ref{prop:dim-F}, we get
  \begin{equation}\label{eq:2.11}
    \langle T_h w_{h,i},w_{h,j}\rangle
    =
    \langle T_h v_{h,i},v_{h,j}\rangle
    +
    \mathcal O(\Lambda_hC_h).
  \end{equation}
  Then, we use
  \[
    \langle T_hv_{h,i},v_{h,j}\rangle
    =
    \langle T_h u_{h,i},u_{h,j}\rangle
    +
    \langle T_h u_{h,i},v_{h,j}-u_{h,j}\rangle
    +
    \langle v_{h,i}-u_{h,i}, T_hv_{h,j}\rangle.
  \]
  By Proposition~\ref{prop:dim-F} and (1) in Assumption~\ref{ass1}, we have
  \[ 
    \begin{aligned}
      \langle T_hu_{h,i},v_{h,j}-u_{h,j}\rangle  
      & =
      \mathcal O(h^{-p}\ee^{-2\mathfrak S_1/h})=\mathcal O(C_h^2)\,,\\
      \langle v_{h,i}-u_{h,i}, T_hv_{h,j}\rangle 
      & =\mathcal O(\Lambda_hC_h).
    \end{aligned}
  \]
  So we get
  \begin{equation}\label{eq:2.12}
    \langle T_hv_{h,i},v_{h,j}\rangle
    =
    \langle T_h u_{h,i},u_{h,j}\rangle + \mathcal O (\varepsilon_h)\,,
    \end{equation}
    which together with~\eqref{eq:2.11} implies~\eqref{eq:inter-c}. 
\end{proof}

An immediate consequence of Proposition~\ref{prop:int-matrix} is an improved lower  bound on the lowest eigenvalue \(\lambda_1(h)\).

\begin{corollary}\label{eq:lb-1st-ev}
  If Assumption~\ref{ass1} holds,   then there exist positive constants \(C,h_1\) such that,  for $h\in (0,h_1]$,
  \begin{equation}\label{eq:lambda-1n} 
    \lambda_1(h)\geq -C \ee^{-\min(\mathfrak{S}_1,2\mathfrak{S}_2,2\mathfrak{S}_3)/h}.
  \end{equation}
\end{corollary}

\begin{proof}
It follows from H\(\ddot{\rm o}\)lder's inequality and (1)--(2) in Assumption~\ref{ass1} that \(\mathsf u_{ij}=\mathcal O(\ee^{-\mathfrak S_1/h})\),  where \(\mathsf u_{ij}\) is introduced in~\eqref{eq:def-u-jk}.   By Proposition~\ref{prop:int-matrix},  we have
\[
  \Lambda_h\leq \|\mathsf W_h\|\leq \|\mathsf U_h\|+\|\mathsf R_h\|
  =
  \mathcal O( \ee^{-\mathfrak S_1/h})+\mathcal (\Lambda_h C_h+C_h^2)
\]
which yields 
\[
  \bigl(1+\mathcal O(C_h)\bigr) \Lambda_h
  =
  \mathcal O( \ee^{-\mathfrak S_1/h})+ C_h^2.
\]
To conclude, we just notice that \(C_h=o(1)\).
\end{proof}

As we  will see in the next subsection, we can actually say much more about the spectrum of these matrices,  when we impose some symmetry condition involving $T_h$ and the choice of the $u_{h,j}$.

\subsection{Implementing invariance assumptions}

Our task in this subsection is to analyze the case when the  matrix \(\mathsf W_h\) of \(T_h|_{F_h}\) enjoys certain invariance properties. We shall see that this corresponds to what occurs in the case of \emph{symmetric} wells in the applications,  starting from the double well case as mathematically considered by E.~Harrell~\cite{Ha} and later extended to the multiple wells case in~\cite{HeSj1,HeSj2,Si1,Si2}. Here we mainly follow in a more abstract way~\cite{HeSj2}  and the heuristic presentation given in~\cite{FH-b}. We denote by $\mathbb Z_n$ the cyclic group of order $n$ and by $\mathfrak g \mapsto \rho(\mathfrak g)$  a faithful unitary representation of $\mathbb Z_n$ in $H$. We denote by $ a_n$ its generator, so $ a_n^n=e$ where $e$ is the identity element of the group.

In addition to the properties in Assumption~\ref{ass1}, we assume
  
\begin{assumption}\label{ass:sym}~
  \begin{enumerate}
     \item 
     The operator \(T_h\) commutes with \(\rho(\mathfrak g)\) for all $\mathfrak g \in \mathbb Z_n$.
     \item 
     \(u_{h,i+1} = \rho(  a_n) u_{h,i} \) for \(1\leq i\leq n-1\).
  \end{enumerate}
\end{assumption}

\begin{remark}\label{remlessabstract}
In the applications considered in this article, the Hilbert space will be $H=L^2(\Omega)$ where $\Omega$ is a domain in $\R^2$. We first consider the unitary representation $\rho_0$ of $\mathbb Z_n$ as the group $G_n$ of the $n$-fold rotations, i.e. the representation such that
\[
  \rho_0 ( a_n)\coloneqq g_n
\]
is the rotation in $\R^2$ by $2\pi/n$ around the origin in $\R^2$.

We let the rotation \(g_n\) act on functions as 
\begin{equation}\label{eq:def-Mgn}
  \bigl(M(g_n)u\bigr)(x)= u ( g_n^{-1} x).
\end{equation}
This gives by extension to any element of $G_n$ a representation of $G_n$ in $L^2(\Omega)$ if $\Omega\subset\R^2$ is a domain invariant by $G_n$ and we then define $\rho$ by
\[
  \rho(\mathfrak g) = M (\rho_0(\mathfrak g )) \,.
\]
Equivalently to Assumption \eqref{ass:sym}, we can then write in this case

\begin{assumption}\label{ass:symbis}~
  \begin{enumerate}
    \item $\Omega\subset\R^2$ is a domain invariant by $G$ and \(H=L^2(\Omega)\).
    \item The operator \(T_h\) commutes with \(M(g_n)\).
    \item \(u_{h,i+1} = M( g_n) u_{h,i}=u_{h,1}(g_n^{-i}x)\) for \(1\leq i\leq n-1\).
  \end{enumerate}
\end{assumption}
Assumption \ref{ass:sym}  permits to treat more general situations which for example occur in the case of manifolds or in the case of higher dimension.
\end{remark}

\begin{remark}\label{rem:reflection}
  If \(n=2\) we can define another group of symmetry \(\tilde G\) defined by the reflection \(\tilde g_2 \begin{psmallmatrix*}[r]x_1\\x_2\end{psmallmatrix*}=\begin{psmallmatrix*}[r]-x_1\\x_2\end{psmallmatrix*}\). We can then consider a variant of Assumption~\ref{ass:sym} by instead assuming that the domain \(\Omega\) is invariant by the reflection \(\tilde g_2\) and that \(u_{h,2}=\overline{M(\tilde g_2)u_{h,1}}\).  This symmetry invariance was assumed   by the papers considering the  magnetic tunneling induced by the geometry of the domain~\cite{Al,BHR,FHK,KR}. Notice that  $M (\tilde g_2)$ does not commute with $T_h$ and that  we have consequently  to compose $M(\tilde g_2)$ with the complex conjugation $\Gamma$ in order to get
  \[
    T_h \Gamma M(\tilde g_2) = \Gamma M(\tilde g_2) T_h \,.
  \]
\end{remark}

\begin{proposition}\label{prop:basis-B}
  If Assumption~\ref{ass:sym} holds, then the orthonormal basis \(\mathcal V_h=\{w_{h,1},\ldots,w_{h,n}\}\) of \(F_h\) satisfies,
  \[
    w_{h,i+1} = \rho( a_n) w_{h,i} \quad(1\leq i\leq n-1).
  \]
\end{proposition}

\begin{proof}
  Recall that \(\{w_{h,1},\ldots,w_{h,n}\}\) is defined in~\eqref{eq:def-fij} by the Gram matrix  
  starting from the basis consisting of the vectors \(v_{h,i}=\Pi_{F_h}u_{h,i}\), \(1\leq i\leq n\). It suffices to observe that  the projector \(\Pi_{F_h}\) on the eigenspace \(F_h\)  commutes with $\rho(\mathfrak g)$. Actually, if \(\{w_{h,1},\ldots,w_{h,n}\}\) is an orthonormal basis of \(F_h\) consisting of eigenvectors \(T_h\), then, since \(T_h\) commutes with  \(\rho( a_n)\), we get that \( \rho( a_n)w_{h,1},\ldots,\rho( a_n)w_{h,n}\) are eigenvectors of \(T_h\) and form an orthonormal basis of \(F_h\).  Consequently
    \[
      \begin{split}
        \Pi_{F_h} \rho( a_n)u
        &=
        \sum_{k=1}^n\langle \rho( a_n)u,\rho( a_n)w_{h,k}\rangle \rho( a_n)w_{h,k}\\
        &=
        \sum_{k=1}^n\langle u,w_{h,k}\rangle \rho( a_n)w_{h,k}=\rho( a_n)\Pi_{F_h}u.\qedhere
      \end{split}
    \]
\end{proof}
  
The matrix of $\rho( a_n)$ in the basis \(\mathcal V_h\) is the same as the matrix of the shift operator $\tau$ on \(\ell^2(\Z/n\Z)\), whose matrix is given by 
\begin{equation}\label{eq:deftau}
  \tau_{j,k} = \delta_{j+1,k} \mbox{ for } 1\leq j,k\leq n
\end{equation} 
where $\delta_{i,k}$ denotes the Kronecker symbol, with \(i\) computed in \(\Z/n\Z\). When $n=2$ and $n=3$, the matrix $\tau$ is respectively given by
 \[ 
  \begin{pmatrix} 
    0 & 1 \\ 
    1 & 0 
  \end{pmatrix}
  \quad\text{and}\quad
  \begin{pmatrix} 
    0 & 1 & 0 \\ 
    0 & 0 & 1 \\ 
    1 & 0 & 0 
  \end{pmatrix}\,.
\]
We observe that 
\[
  \tau^{n-1}=\tau^{-1}=\tau^*\,.
\]
The property that the operator $T_h$ commutes with $\rho( a)$ implies that the matrix \(\mathsf W_h\) (of \(T_h|_{F_h}\) in the basis $\mathcal V_h$) commutes with \(\tau\), i.e., \(\tau \mathsf  W_h = \mathsf  W_h  \tau\). Note that this invariance condition yields that 
\begin{equation}\label{eq:Mh-exp}
  \mathsf W_h = \sum_{k=0}^{n-1} I_k(h)  \tau^k,
\end{equation}
for some coefficients $I_0(h), \ldots, I_{n-1}(h) \in \C$. Here \(\tau^0\) denotes the identity matrix.

The Hermitian property of $\mathsf{W}_h$ gives, in addition,
\begin{equation}\label{eq:saadj}
  I_0(h) \in \mathbb R\,,\, 
  I_k(h) = \overline {I_{n-k}(h)}\;\mbox{ for } k=1,\dots, n-1\,.
\end{equation}

Notice that the matrix \(\mathsf U_h\) introduced in Proposition~\ref{prop:int-matrix} satisfies the same properties as $\mathsf W_h$. Hence we can also write
\begin{equation}
  \mathsf U_h = \sum_{k=0}^{n-1} J_k(h) \tau^k \,,
\end{equation}
for some coefficients $J_0(h), \ldots,J_{n-1}(h) \in \C\,$ and the Hermitian property of $\mathsf U_h$ also implies
\begin{equation}\label{eq:saadjA}
  J_0(h) \in \mathbb R\,,\, 
  J_k(h) = \overline {J_{n-k}(h)}\;\mbox{ for } k=1,\dots, n-1\,.
\end{equation}

All these invariant matrices (\(\mathsf W_h\) or \(\mathsf U_h\)) share the property to be diagonalizable in the same orthonormal basis of eigenfunctions $\ee_k$ $(k=1,\dots,n)$ whose coordinates in our selected basis are given by
\[
  (\ee_k)_\ell = \omega_n^{(k-1) \ell}\,,
  \quad\text{with}\quad
  \omega_n\coloneqq  \exp (2 i \pi/n)\,.
\]
It is then easy to compute the corresponding eigenvalues.

In particular, we get an explicit representation of the eigenvalues of \(\mathsf W_h\) which  illustrates when $n=3,4$,  the possibility of eigenvalue crossings (i.e. change of multiplicity).
\begin{itemize}
  \item When \(n=2\),  \(\mathsf W_h\) assumes the form
  \(
    \begin{pmatrix}
      I_0&I_1\\
      I_1&I_0
    \end{pmatrix}
  \)
  with \(I_1\) real. This matrix has two eigenvalues
  \begin{equation}\label{eq:formn=2}
    \lambda_1 = I_0 - |I_1|\,,\, \lambda_2= I_0 + |I_1|\,.
  \end{equation}
  \item When \(n=3\),  \(\mathsf W_h\) assumes the form
  \[
  \mathsf{W}_h=
    \begin{pmatrix}
      I_0&\overline{I_1}&I_1\\
      I_1&I_0& \overline{I_1}\\
      \overline{I_1}&I_1&I_0
    \end{pmatrix}
  \]
  with \(I_1=\rho\ee^{\ii\theta}\), \(\rho \geq 0\), \(\theta \in [0,2\pi)\). This matrix has three eigenvalues
  \begin{equation}\label{eq:formn=3}
    \mu_k = I_0 + 2 \rho\cos\left(\theta+(k-1)\frac{2\pi}{3}\right),\quad k\in \{1,2,3\}.
  \end{equation}
  \item When \(n=4\), we meet the matrix
  \[ \mathsf{W}_h=
    \begin{pmatrix}
      I_0&  \overline{I_1} &I_2& I_1 \\
      I_1& I_0 & \overline{I_1}& I_2  \\
      I_2 & I_1 & I_0 & \overline{I_1} \\
      \overline{I_1}&I_2& I_1 &I_0 
    \end{pmatrix}
  \] 
  with \(I_2\) real and \(I_1=\rho\ee^{\ii\theta}\), \(\rho \geq 0\), \(\theta \in [0,2\pi)\). We refer to~\cite{FH-b} for a  further discussion of this case.  Figure~\ref{fig2} illustrates the braid startucture of the eigenvalues of the matrix \(\mathsf W_h\).
  \begin{figure}[htb]
     \includegraphics[page=6]{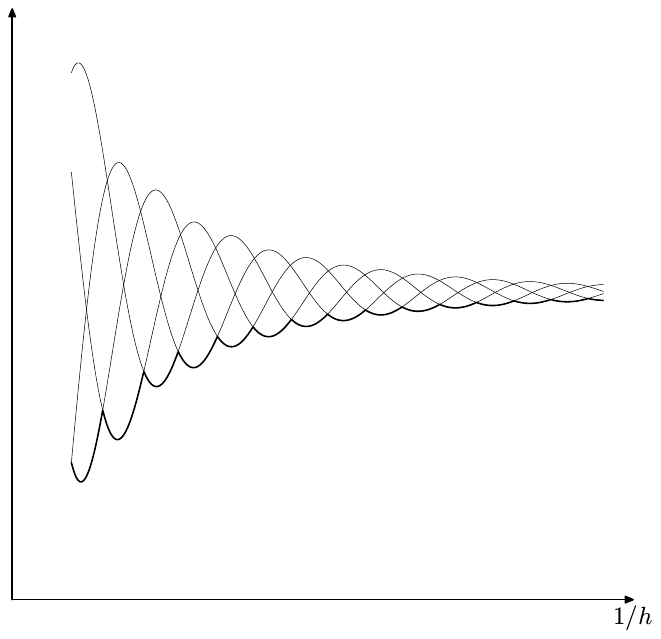}
     \caption{A schematic figure of four eigenvalues with a braid structure in the case of a rotational symmetry of order 4. } 
     \label{fig2}
   \end{figure}
\end{itemize}

\subsection{Applications}

\subsubsection{Additional hypothesis}\label{rem:optimal-error}
We can strengthen the  estimate of \(\Lambda_h\) in~\eqref{eq:def-Lambda-h} if we assume additionally that\footnote{We shall see in the applications that the rough estimate of \(J_k(h)\) by H\(\rm\ddot o\)lder's inequality and (1)--(2) in Assumption~\ref{ass1} is not sufficient to the accurate estimate of tunneling.}

\begin{assumption}\label{ass-error}
There exists a positive  constant \(\mathfrak S\) such that
\begin{equation}\label{eq:newcond}
  \mathfrak S < 2\min_{1\leq j\leq 3} \mathfrak S_j\,,
\end{equation}
\begin{subequations}\label{eq:assnew1}
  \begin{equation}
    J_{k}(h) \underset{h\searrow0}{=}\mathcal O(\ee^{- \mathfrak S/h}) \quad  (k=0,\cdots,n-1),
  \end{equation}
  and 
  \begin{equation}
    |J_1(h)|\underset{h\searrow0}{=}\ee^{-(\mathfrak S+o(1))/h}.
  \end{equation}
\end{subequations}
\end{assumption}

\begin{proposition}\label{prop:control-Lambda-h}
There exist positive constants \(C,h_0>0\) such that, if Assumptions~\ref{ass1} and \ref{ass-error} hold,  then for all \(h\in(0,h_0]\), the symmetric matrix \(\mathsf{R}_h=(\mathsf r_{ij})\) introduced in~\eqref{eq:inter-a} satisfies
\[ 
  \|\mathsf R_h\|
  \underset{h\searrow0}{=} 
  \mathcal O\big(\ee^{-3\mathfrak S/2h}\big)=o\big(|J_1(h)|\big).
\]
\end{proposition}

\begin{proof}
This follows by applying Proposition~\ref{prop:int-matrix} (the same argument as in Corollary~\ref{eq:lambda-1n}).  
Indeed,  we have by the first identity in~\eqref{eq:assnew1}, 
\begin{equation}\label{eq:imp1}
  \Lambda_h=\mathcal O(\ee^{-\mathfrak S/h})+\mathcal O(C_h^2)\,.
\end{equation}
Inserting~\eqref{eq:imp1} into the definition of  \(\varepsilon_h\) in~\eqref{eq:def-Lambda-h}, we obtain that
\begin{equation}\label{eq:def-delta-h}
  \varepsilon_h
  =
  \mathcal O(\delta_h)\mbox{ where }\delta_h=C_h\ee^{-\mathfrak S/h}+C_h^2
  =
  \mathcal O(\ee^{- 3\mathfrak S/2h}).
\end{equation}
We can improve the bounds~\eqref{eq:inter-c} of the symmetric matrix \(\mathsf{R}_h\) into \(\mathsf r_{ij} = \mathcal O(\delta_h)\). To finish the proof,  we observe that,  by the second identity in~\eqref{eq:assnew1}, \( \delta_h=o(|J_1(h)|)\).
\end{proof}

\subsubsection{The case $n=2$}\label{sec:2-wells}
A first consequence of the previous  analysis is a full understanding of  the case corresponding to   $n=2$ where the symmetry $g_2$ reads $(x_1,x_2) \mapsto (-x_1,-x_2)$. Assuming that Assumptions~\ref{ass:symbis} and \ref{ass-error} hold, we get from~\eqref{eq:formn=2} and Proposition~\ref{prop:control-Lambda-h}
\begin{equation}\label{eq:tun-error-a} 
  \lambda_2(h)-\lambda_1(h)=2 |J_1(h) |+ o\big(|J_1(h)|\big).
\end{equation}

\subsubsection{The case $n=3$, braid structure of eigenvalues}\label{sec:3-wells}
Suppose that Assumptions~\ref{ass:symbis} and \ref{ass-error} hold with \(n=3\).  Let  \(I_0(h),I_1(h)\) be as in~\eqref{eq:Mh-exp} and let us write 
\[
  I_{1}(h)=\rho(h) \ee^{\ii\theta(h)}
  \text{  where } \rho(h)\geq 0 \text{ and } \theta(h)\in[0,2\pi).
\]
Then, by~\eqref{eq:formn=3}, we have a relabeling \(\mu_1(h),\mu_2(h)\), and \(\mu_3(h)\) of the eigenvalues \(\lambda_1(h),\lambda_2(h)\), and \(\lambda_3(h)\) of $T_h$ with $\rho=\rho(h)$ and $\theta=\theta(h)$. Moreover,
\begin{subequations}\label{eq:Ik=Jk}
  \begin{equation} 
    I_0(h)= J_0(h) +\mathcal O(\delta_h),\quad 
    I_1(h) = J_1(h)+\mathcal O(\delta_h)\sim J_1(h)
  \end{equation}
  with \(\delta_h=o\big(J_1(h)\big)\) defined in~\eqref{eq:def-delta-h} and
  \begin{equation}
    J_0(h)=\langle T_hu_{h,1},u_{h,1}\rangle,\quad 
    J_1(h)= \langle T_hu_{h,1},u_{h,2}\rangle.
  \end{equation}
\end{subequations}
Notice that there is possibility for eigenvalue crossings between 

\begin{itemize}
  \item \(\mu_1(h)\) and \(\mu_2(h)\) if \(\theta(h) \in \{ 2\pi/3, 5\pi/3 \}\);
  \item \(\mu_2(h)\) and \(\mu_3(h)\) if \(\theta(h) \in \{ 0, \pi \}\);
  \item \(\mu_1(h)\) and \(\mu_3(h)\) if \(\theta(h) \in \{ \pi/3, 4\pi/3 \}\).
\end{itemize}

The point is then to seek an accurate approximation of \(\theta(h)\).  Notice that~\eqref{eq:Ik=Jk} yields \( I_1(h)\underset{h\searrow0}{\sim}J_1(h)\). Defining \(\theta_1(h)\) by \(J_1(h) =\rho_1 (h) \ee^{ i \theta_1 (h)}\), we can approximate \(\theta (h)\) by \(\theta_1(h)\) modulo \(2\pi\Z\). This could confirm the predicted braid structure mentioned in~\cite{FH-b,BDMV}. We will consider in detail three models where such a phenomenon holds (see Theorem~\ref{thm:HK-3wells}, Section~\ref{sec.Neumann}, and Section~\ref{sec.edge}). 

\section{Electro-magnetic tunneling}\label{sec.HK22}

\subsection{Introduction}

In this section, the  Hilbert space is \(H=L^2(\R^2)\),  and for given \(n\in\mathbb{N}\) we consider \(g = g_n\) to be rotation around the origin by \(2\pi/n\), and \(M(g)\) to be as in~\eqref{eq:def-Mgn}.  We are interested in the spectrum of the electro-magnetic Schr\"odinger operator
\[
  \mathcal L_{h,b}=(-\ii h\nabla-b\Ab)^2+V,
\]
where \(b,h>0\), and 
\begin{equation}\label{eq:A(x)}
  \Ab(x)=\frac12(-x_2, x_1)\,.
\end{equation}
Notice that \(\Ab\) generates the constant magnetic field \(\curl\Ab=1\). For the potential $V$ we assume that
\begin{subequations}
  \begin{equation}
    V\in C^\infty(\R^2,\R)\,,
  \end{equation}
  \begin{equation}\label{eq:V<0}
    V\leq 0\,,
  \end{equation}
  \begin{equation}\label{eq:sym}
    V \text{ is invariant by the rotation } g_n\,.
  \end{equation}
  Moreover, we assume that
  \begin{equation}
    \text{The minimum of } V \text{  is attained at } n  \text{ non-degenerate minima.}
  \end{equation}
\end{subequations}
Then it results from the invariance property of $V$ that  these minima are  \(n\) equidistant points of \(\R^2\setminus \{0\}\). We will refer to these points as the \emph{wells}.
 
Notice that, when dealing with a fixed \(b>0\) we can reduce  the analysis  to the case where \(b=1\) by introducing an effective semi-classical  parameter \(\hbar=b^{-1}h\) so that
\begin{equation}\label{eq:red-b=1}
  \mathcal L_{h,b}
  =
  b^2\bigl((-\ii \hbar\nabla-\Ab)^2+b^{-2}V\bigr).
\end{equation}
So we will assume henceforth that \(b=1\). To relate with the discussion in Section~\ref{sec.IM}, our operator \(T_h\) is  the electro-magnetic Laplacian shifted by a certain scalar \(\lambda(h)\).
\begin{equation}\label{eq:H-h}
  T_h
  \coloneqq 
  \mathcal L_h-\lambda(h),\quad 
  \mathcal L_h = (-\ii h\nabla-\Ab)^2 + V.
\end{equation}
Note that the  assumption in~\eqref{eq:sym} implies that $T_h$ commutes with $M(g)$. Hence the condition of invariance of the preceding section holds.
The shift constant \(\lambda(h)\) in~\eqref{eq:H-h} will be chosen as the ground state energy of a reference single well operator. 

\subsection{Single well ground states} 

Let us first discuss the one well operator.

\subsubsection{Preliminary discussion and assumptions}
There are various possible approaches to create one well problems in the presence of multiple wells.  The approach considered in~\cite{HeSj1,HeSj2}  starts from a general electric potential \(V\) and introduce suitable Dirichlet conditions to create an infinite barrier leading to  a one well problem. The so-called LCAO\footnote{for Linear Combination of Atomic Orbitals.} approach,  frequently used in Physics and Atomic Chemistry,   and considered in~\cite{FSW,HK22}, particularly applies when \(V\) is  a superposition of single well potentials.
 
As in~\cite{FSW,HK22},  we consider a radial single well potential.  More precisely,  we assume in this section that \(V=\mathfrak v_0\),  where  $\mathfrak v_0\in C_c^\infty(\R^2)$  is a non-positive radial function satisfying
\begin{equation}\label{eq:v0}
  \begin{cases}
  \mathfrak v_0(x)=v_0(|x|)
  ~\&~
  v_0^{\min}\coloneqq \min\limits_{r\geq 0}v_0(r)<0\,,\\
  \supp \mathfrak v_0\subset  \overline{D(0,a)}\coloneqq \{x\in\R^2\,:\,|x|\leq a\}\,,\\
  \mathfrak{v}_0^{-1}(v_0^{\mathrm{min}}) = \{0\}
  \quad\&\quad v_0''(0)>0\,.
\end{cases}
\end{equation}
We choose $\overline{D(0,a)}$ as the smallest  closed disc containing ${\rm supp}\,\mathfrak v_0$, i.e.
\begin{equation}\label{eq:def-a}
  a = a(v_0)\coloneqq \inf\{r>0~:~v_0|_{[r,+\infty)}=0\}.
\end{equation}

\subsubsection{Preliminary inequalities}
As in~\cite{HK22},   we  will encounter various errors of exponential order which are defined in terms of  the function \(v_0\) and the  constant  
\begin{equation}\label{eq:L2a} 
  L > 2a(v_0)\,.
\end{equation}
With $(v_0,L)$  as above,  \(a=a(v_0)\) and with
\begin{equation}\label{eq:defd}
  d(r)
  =
  d_{v_0} (r)
  \coloneqq 
  \int_0^{r}\sqrt{\frac{\rho^2}{4}+v_0(\rho)-v_0^{\min}}\dd\rho,
\end{equation}
we introduce the four constants (we will often skip the depends on \(v_0\) and \(L\))
\begin{equation}\label{eq:defSs}
  \left.
  \begin{aligned}
    S_0 = S_0(v_0,L)&\coloneqq d (L),\\
    S_a =S_a(v_0,L) &\coloneqq  d (a) + d (L-a),\\
    \hat{S}_a =\hat S_a(v_0,L)&\coloneqq  d (L-a), \\
    \hat S=\hat S(v_0,L)&\coloneqq \inf_{0<r<a}\biggl[ \frac{Lr}{2}+d (r)+d (L-r)\biggr].
  \end{aligned}
  \right\}
\end{equation}
Notice also that, since \(v_0\) vanishes on \([a,+\infty)\),  we can  express the constant \(S_a\) as follows
\[
  S_a
  =
  2\int_0^a\sqrt{\frac{r^2}{4}+v_0(r)-v_0^{\min}}\dd r+\int_0^{L-a}\sqrt{\frac{r^2}{4}-v_0^{\min}}\dd r.
\]
It is important when comparing the errors  to compare the three constants introduced in~\eqref{eq:defSs}. 

\begin{proposition}\label{prop:comp-SaS0}
Assume that \( v_0\) and \(L\) satisfy~\eqref{eq:def-a} and~\eqref{eq:L2a}.  Then we have, with $a=a(v_0)$,
\begin{equation}\label{eq:3.8}
  \hat S_a(v_0,L)
  <
  S_a(v_0,L)
  <
  \hat S(v_0,L)
  <
  \min\Bigl( S_0(v_0,L),\frac{La}{2}+S_a(v_0,L)\Bigr).
\end{equation}
Moreover,  if \(v_0\) and \(L\) satisfy \(L\geq 4a(v_0)\),  then \(2\hat S_a(v_0,L)>\hat S(v_0,L)\).
\end{proposition}

\begin{proof}
The inequality \(\hat S_a<S_a\) is obvious. The other  inequalities  in~\eqref{eq:3.8} are proved in~\cite[Prop.~3.5]{HK22}. So assume that \(L>4a\) and let us prove that \(2\hat S_a>\hat S\).  Notice that \(S_a=\hat S_a+d(a)\) and \(\hat S<\frac{La}{2}+S_a\),  hence
\[
  2\hat S_a-\hat S
  >
  2\hat S_a-S_a-\frac{La}{2}
  =
  \int_a^{L-a}\underset{\geq\rho/2}{\uwave{\sqrt{\frac{\rho^2}{4}+v_0(\rho)-v_0^{\min}}}}\dd \rho-\frac{La}2
  \geq
  \frac{L^2-4aL}{4}>0.
\]
\end{proof}

\subsubsection{The one well approximate eigenfunction}

Consider the single well operator
\begin{equation}\label{eq:op-1well}
  \mathcal L_h^{\rm sw}
  =
  (-\ii h\nabla-\Ab)^2+\mathfrak v_0
\end{equation}
whose ground state energy is 
\begin{equation}\label{eq:gse-1well}
  \lambda(h)
  =
  \inf\sigma(\mathcal L_h).
\end{equation}
Since \(\mathfrak v_0\leq 0\), \(\mathcal L_h^{\rm sw}\) is smaller (in the sense of comparison of self-adjoint operators) than the Landau Hamiltonian \((-\ii h\nabla-\Ab)^2\). Hence  we have by the min-max principle  
\begin{equation}\label{eq:3.11a}
  \lambda(h)\leq h\,.
\end{equation}

In light of the conditions on \(v_0\) in~\eqref{eq:v0}, we know from~\cite[Thm.~1.1]{HK22} that \(\lambda(h)\) is a simple eigenvalue and that  \(\mathcal L_h^{\rm sw}\) has a positive radial ground state satisfying, for any  relatively compact domain $K$ of \(\R^2\), 
\begin{equation}\label{eq:d-main0}
  \mathfrak u_h(x)=u_h(|x|),\quad 
  \int_{\R^2}|\mathfrak u_h(x)|^2\dd x=1,\quad 
  \norm{ e^{d(\abs{x})/h}\mathfrak  u_h(x) }_{H^2(K)}=\mathcal O(h^{-1/2}),
\end{equation}
 
An important term in the context of tunneling is  the hopping coefficient defined  in~\cite{FSW, HK22} (in the case $n=2$) by
\begin{equation}\label{eq:hopping}
  \mathfrak c_h(v_0,L)
  =
  \int_{D(0,a)} \mathfrak v_0(x)\mathfrak u_h(x)\mathfrak u_h(x_1+L,x_2)\ee^{\frac{\ii Lx_2}{2h}}\dd x
\end{equation}
where \(a=a(v_0)\) (see~\eqref{eq:def-a}).

As observed in~\cite{FSW}, the hopping coefficient is a negative real number and an accurate estimate of it was established recently  in~\cite[Sec.~6 \& Eq.~(6.25)]{HK22},  which we recall below.\footnote{Notice that the estimate of \(\mathfrak c_h(v_0,L)\) does not require the assumption \(L>4\big(a(v_0) +\sqrt{v_0^{\min}}\big)\) imposed in~\cite{FSW} (see also~\cite[comments after Eq.~(5.2)]{HK22}).} 

\begin{proposition}\label{prop:HK22-hopping}
Assume that \(\mathfrak v_0\) and \(L\) satisfy~\eqref{eq:v0} and~\eqref{eq:L2a}.  Then there exists a positive constant \(S(v_0,L)\) such that
\begin{equation}\label{eq:3.14-a}
  \lim_{h\searrow0}\bigl(h\ln|\mathfrak c_h(v_0,L)|\bigr)=-S(v_0,L).
\end{equation}
Moreover,   
\begin{equation}\label{eq:3.14a}
  S_a(v_0,L)<S(v_0,L)<\hat S(v_0,L)
\end{equation}
where \(S_a(v_0,L)\) and \(\hat S(v_0,L)\) are introduced in (\ref{eq:defSs}).
\end{proposition}  

We shall use Proposition~\ref{prop:HK22-hopping} later on in the proof of  Proposition~\ref{prop:HK-sym} when dealing with \(n\) potential wells ($n\geq 2$). Let us recall the  definition of \(S(v_0,L)\) given in~\cite[Eq.~(6.18) and (6.20)]{HK22}. 
We have 
\begin{equation}\label{eq:def-S(v0)**}
  S(v_0,L)
  \coloneqq 
  -F(\mathfrak  v_0)+\inf_{\substack{r\in[0,a]\\ t\in(0,+\infty)}}\Psi(r,t),
\end{equation}
where $a=a(v_0)$ is introduced in~\eqref{eq:def-a} and, with $d$ introduced in~\eqref{eq:def-a}, 
\begin{equation}\label{eq:def-Psi}
  \left.
  \begin{aligned}
    F(v_0)
    &\coloneqq 
    \frac{a}{4}\sqrt{a^2+4|v_0^{\min}|} + \frac{|v_0^{\min}|}{2} \ln\frac{\big(\sqrt{a^2+4|v_0^{\min}|}+a\big)^2 }{4|v_0^{\min}|}-d(a),\\
    \Psi(r,t)
    &\coloneqq 
    d(r)+\frac{r^2+L^2}{4}(2t+1)+\frac{|v_0^{\min}|}{2}\ln\left(1+\frac1{t}\right)-Lr\sqrt{t(t+1)}.
  \end{aligned}
  \right\}
\end{equation}
The following proposition deals with a term similar to the hopping coefficient and plays a key role in the approximation of various error terms that we will encounter later, e.g.  when we verify Assumption~\ref{ass1} for an electro-magnetic Schr\"odinger operator with multiple wells, as in Proposition~\ref{prop:HK-ass1}.  

\begin{proposition}\label{lem:HK-int}
We have, as $h\to 0$,
\[
  w
  \coloneqq 
  \int_{D(0,L)} \mathfrak u_h(x)\mathfrak u_h(x_1+L,x_2)\ee^{\frac{\ii Lx_2}{2h}}\dd x 
  =
  \mathcal O\left(\ee^{\frac{-S_a+o(1)}{h}}\right),
\]
where $a=a(v_0)$,  $L >2a$ and $S_a$ is defined in~(\ref{eq:defSs}b).
\end{proposition}

\begin{proof}
We can estimate \(w\) in the same way as the hopping coefficient \(\mathfrak c_h(\mathfrak v_0,L)\) was estimated in~\cite[Prop.~5.1]{HK22}. 

The only difference is that in the expression of \(\mathfrak c_h(v_0,L)\) (see~\eqref{eq:hopping})  we encounter the potential energy term \(v_0\) in the integrand and the integral is consequently over the smaller disc \(D(0,a)\).  Hence the new term to estimate corresponds to $w_2$  below, $w_3$ being of the same type as $w_1$ after a symmetry.

We decompose the integral defining \(w\) into three terms
\[
  \begin{aligned}
    w&=w_1+w_2+w_3\,,\\
    w_1&=\int_{D(0,a)} \mathfrak u_h(x)\mathfrak u_h(x_1+L,x_2)\ee^{\frac{\ii Lx_2}{2h}}\dd x\,, \\
    w_2&=\int_{D(0,L-a)\setminus D(0,a)} \mathfrak u_h(x)\mathfrak u_h(x_1+L,x_2)\ee^{\frac{\ii Lx_2}{2h}}\dd x\,, \\
    w_3&=\int_{D(0,L)\setminus D(0,L-a)} \mathfrak u_h(x)\mathfrak u_h(x_1+L,x_2)\ee^{\frac{\ii Lx_2}{2h}}\dd x\,. 
  \end{aligned}
\]

\textbf{Step 1: Contribution of the integral in  \(D(0,a)\).} We express the integral defining \(w_1\) in polar coordinates
\begin{equation}\label{eq:w-int1}
  w_1
  =
  \int_0^{2\pi}\int_0^{a} u_h(r) u_h\bigl(\sqrt{r^2+L^2+2Lr\cos\theta}\bigr) \ee^{\frac{\ii Lr\sin\theta}{2h}} r\dd r.
\end{equation}
In light of the decay property in~\eqref{eq:d-main0}, we get
\[
  w_1=\mathcal O(\ee^{-\tilde S/h}),
\]
where
\[
  \tilde S
  =
  \inf_{0<\theta<2\pi}
  \Big\{
    \inf_{0<r<a}
      \bigl[
        d(r)+d(\sqrt{r^2+L^2+2Lr\cos\theta})
      \bigr]
  \Big\}.
\]
Since \(\sqrt{r^2+L^2+2Lr\cos\theta}>L-r\) and since \(r\mapsto d(r)\) is non-decreasing,
\[
  \tilde S
  \geq
  \inf_{0<r<a}\bigl(d(r)+d(L-r) \bigr)
  =
  d(a) + d(L-a)
  =
  S_a
\]
where \(S_a\) is introduced in (\ref{eq:defSs}b). Notice that the penultimate identity follows from the fact that
\[
  \frac{\dd}{\dd r}[d(r) + d(L-r)]
  =
  \sqrt{\frac{r^2}{4}+v_0(r)-v_0^{\min}} 
  - \sqrt{\frac{(L-r)^2}{4}-v_0^{\min}} 
  < 
  0 
  \quad (0<r<a).
\]
Therefore, we have that
\begin{equation}\label{eq:3.16}
  w_1
  =
  \mathcal O(\ee^{-S_a/h}).
\end{equation}

\textbf{Step 2: Contribution of the integral in  \(D(0,L)\setminus D(0,L-a)\).} We argue as in Step 1. Expressing the integral defining \(w_3\) in polar coordinates
\begin{equation}\label{eq:w-int3}
  w_3
  =
  \int_0^{2\pi}\int_{L-a}^{L}  u_h(r) u_h\bigl(\sqrt{r^2+L^2+2Lr\cos\theta}\bigr) \ee^{\frac{\ii Lr\sin\theta}{2h}} r\dd r,
\end{equation}
we get
\[
  w_3=\mathcal O(\ee^{-\tilde S'/h}),
\]
where
\[
  \begin{aligned}
    \tilde S'&=\inf_{0<\theta<2\pi}\left(\inf_{L-a<r<L}\Bigl(d(r)+d\bigl(\sqrt{r^2+L^2+2Lr\cos\theta}\bigr) \Bigr)\right)\\
    &\geq \inf_{L-a<r<L}\bigl(d(r)+d(L-r) \bigr)
    =\inf_{0<t<a}\bigl(d(t)+d(L-t) \bigr)=S_a.
  \end{aligned}
\] 
Therefore, we have that
\begin{equation}\label{eq:3.18}
  w_3 = \mathcal O(\ee^{-S_a/h}). 
\end{equation}

\textbf{Step 3: Contribution of the integral in   \(D(0,L-a)\setminus D(0,a)\).} We express the integral defining \(w_2\) as follows
\begin{equation}\label{eq:def-w2}
  w_2
  =
  \int_a^{L-a}  u_h(r) \biggl(\int_0^{2\pi}K_h(r,\theta)d\theta\biggr)r\dd r\,, 
\end{equation}
where
\[
  K_h(r,\theta)= u_h\left(\sqrt{r^2+L^2+2Lr\cos\theta}\right) e^{\frac{\ii Lr\sin\theta}{2h}}\,.
\] 
Observing that for \(a<r<L-a\), we have \(a<L-r<L-a\) and
\[
  \rho\coloneqq \sqrt{r^2+L^2+2Lr\cos\theta}\geq L-r>a,
\] 
so \(u_h(\rho)\) has a nice integral representation~\cite[Eq.~(2.9)]{FSW}. Consequently,  the integral of $K_h$ with respect to  $\theta$ is computed as in~\cite[Prop.~5.1]{FSW}.  In fact,  we have    (see~\cite[Eq.~(5.9)-(3.10)]{HK22})
\begin{equation}\label{eq:int-Kh} 
  \int_0^{2\pi}K_h(r,\theta)d\theta 
  =
  C_h\exp\left( -\frac{r^2+L^2}{4h}\right)\int_0^{+\infty} G_h(r,t) \dd t\,,
\end{equation}
where 
\begin{equation}\label{eq:def-Gh}
  G_h(r,t)
  =
  \exp\left( -\frac{(r^2+L^2)t}{2h} \right)
  t^{\alpha-1}(1+t)^{-\alpha}I_0\left(\frac{Lr\sqrt{t(t+1)}}{h}\right),
\end{equation}
\(C_h,\alpha\) are constants and  $z\mapsto I_0(z)$ is the modified Bessel's function of order $0$ (see~\cite[Lem.~5.2]{HK22}).

Let \(\eta\in(0,1)\) and observe that~\cite[Lem.~6.1]{HK22},
\begin{multline*}
  \int_a^{L-a}  u_h(r) \biggl(C_h\exp\biggl( -\frac{r^2+L^2}{4h}\biggr)\int_0^{\eta} G_h(r,t) dt\biggr)r\dd r \\
  =
  \mathcal O(\ee^{c\sqrt{\eta}/h})\int_a^{L-a} u_h(r)u_h(\sqrt{L^2+r^2}) r\dd r. 
\end{multline*}
By the decay properties of \(u_h\) in~\eqref{eq:d-main0}, we get  (using also~\eqref{eq:3.8} for the last estimate)
\begin{align*}
  \int_a^{L-a}  u_h(r) \biggl(C_h\exp\biggl( -\frac{r^2+L^2}{4h}\biggr)\int_0^{\eta} G_h(r,t) dt\biggr)r\dd r 
  & =
  \mathcal O(\ee^{(c\sqrt{\eta}-S_0)/h}) \\
  & =\mathcal O(\ee^{-S_a/h}),
\end{align*}
for sufficiently small \(\eta\).

Now we deal with the following integral 
\[
  \int_a^{L-a}  u_h(r) \biggl(C_h\exp\biggl( -\frac{r^2+L^2}{4h}\biggr)\int_{\eta}^{+\infty} G_h(r,t) dt\biggr)r\dd r,
\]
which is asymptotically equivalent to~\cite[Lem.\ 6.4]{HK22}
\[
  \frac{\mathfrak  m( v_0)}{\sqrt{2\pi h}}\int_\eta^a \sqrt{r}\,a_0(r)\int_\eta^{+\infty} g_0(t)\exp\biggl(-\frac{\Psi(r,t)-F(v_0)}{h}\biggr)\dd t\dd r,
\]
where \(F(v_0),\Psi(r,t)\) are introduced in~\eqref{eq:def-Psi} and $\mathfrak m(v_0),g_0(t)$ are introduced in~\cite[Eq.\ (5.3) and (6.5)]{HK22} as follows
\[
  \begin{aligned}
    \mathfrak{m}(v_0)
    &= 
    \frac{\sqrt{1+2v_0''(0)}}{2\pi}\sqrt{\frac{2a|v_0^{\min}|}{\pi}} \frac{\left(a^2+4|v_0^{\min}| \right)^{1/4} }{\sqrt{a^2+4|v_0^{\min}|}+a},\\
    g_0(t)
    &=
    \frac1{t^{5/4}(t+1)^{1/4}}
    \left(1+\frac1{t}\right)^{\frac12(\sqrt{1+2v_0''(0)}-1)}.
  \end{aligned}
\]
Of importance to us is that
\[
  -F(\mathfrak v_0)+\inf_{(r,t)\in[a,L-a]\times\R_+}\Psi(r,t)
  \geq
  S_a,
\]
which follows by the same argument as used in the proof of~\cite[Prop.~6.5]{HK22}; for convenience,  we provide details in Appendix~\ref{sec:A}.

We get eventually 
\begin{equation}\label{eq:3.23}
  w_2
  =
  \mathcal O(\ee^{(-S_a+o(1))/h}).
\end{equation}
With~\eqref{eq:3.16} and~\eqref{eq:3.18},  this achieves the proof of the proposition.
\end{proof}
  
\subsection{Verifying Assumption~\ref{ass1}---superposition of single well potentials}

We study the specific case where  the potential  \(V\) is given by
\begin{equation}\label{eq:V}
  V(x)
  =
  \sum_{k=1}^n \mathfrak v_0(x-z_k)\,,
\end{equation}
where  \(\mathfrak v_0\) is the non-positive radial function satisfying~\eqref{eq:v0} and (we identify $\mathbb C$ and $\mathbb R^2$)
\begin{equation}\label{eq-def-z}
  z_k
  \coloneqq
  \frac{L}{\sqrt{2-2\cos(2\pi/n)}}\ee^{2\ii k\pi/n},\quad L>0.
\end{equation}
The wells in this setting are the points \(\{z_k\}_{1\leq k\leq n}\) which are selected so that
\begin{equation}\label{eq:dist-z}
  \dist(z_k,z_{k+1})
  =
  L.
\end{equation}
Let us verify that $V$ satisfies~\eqref{eq:sym}. Our construction of  the points \(z_k\) is such that \(z_{k+1}=g z_k\), for \(k\in\Z/n\Z\). Since \(\mathfrak v_0\) is radial, we have
\[
  \mathfrak v_0(g^{-1}x-z_k)
  =
  v_0(|g^{-1}(x-g z_k)|)
  =
  \mathfrak v_0(x-z_{k+1}).
\]
Consequently,~\eqref{eq:sym} holds and  \(T_h\) commutes with \(M(g)\).  We still have to check that Assumption~\ref{ass1} holds with the choice in Assumption~\ref{ass:symbis} (3).

We introduce the functions
\begin{equation}\label{eq:def-un-HKa}
  u_{h,k}(x)
  =
  \chi(x-z_k)\mathfrak u_h(x-z_k)\ee^{-\ii z_k\cdot\Ab(x)/h}\quad(1\leq k\leq n),
\end{equation}
where 
\begin{itemize}
  \item \((z_k)_{1\leq k\leq n}\) are the points introduced in~\eqref{eq-def-z};
  \item \(\Ab\) is the vector field introduced in~\eqref{eq:A(x)};
  \item \(\chi\in C_c^\infty(\R^2;[0,1])\) is a  radial cut-off function satisfying \(\chi=1\) on \(D(0,L)\) and \(\supp \chi\subset D(0,L+\eta)\)   with \(\eta\in(0,1)\)  fixed arbitrarily.
\end{itemize}
The phase term in~\eqref{eq:def-un-HKa} is due to the effect of \emph{magnetic} translation, which ensures that \(\psi_{h,k}(x)=\mathfrak u_h(x-z_k)\ee^{-\ii z_k\cdot\Ab(x)/h}\) satisfies
\begin{equation}\label{eq:Lh-psi}
  \mathcal L_{h,k}^{\rm sw} \psi_{h,k}
  =
  \lambda(h)\psi_{h,k}
  \quad \text{where}\quad
  \mathcal L_{h,k}^{\rm sw}
  \coloneqq (-\ii h\nabla-\Ab)^2+\mathfrak v_0(\cdot-z_k).
\end{equation}
The above constructions ensures that Assumption~\ref{ass:symbis} (3) holds (since  \(z\cdot\Ab(g_n^{-1}x)=(g_nz)\cdot\Ab(x)\)).  Moreover, the next proposition shows that Assumption~\ref{ass1} holds too.

\begin{proposition}\label{prop:HK-ass1}
Let \(T_h\) and \(\lambda(h)\) be as  in~\eqref{eq:H-h} and~\eqref{eq:gse-1well} respectively.   The conditions in Assumption~\ref{ass1} hold with the following choices: 
\begin{enumerate}[\rm (a)]
  \item  \(\{u_{h,1},\ldots,u_{h,n}\}\) are as in~\eqref{eq:def-un-HKa};
  \item any constants \(\mathfrak S_1,\mathfrak S_2,\mathfrak S_3,p,q\)  satisfying
  \begin{equation}\label{eq:condhatS}
    \mathfrak S_1\in(0,\hat S_a),\quad \mathfrak S_2\in (0, 2\hat S_a), \quad \mathfrak S_3\in (0, \hat S_a ),\quad  p\in(0,1],\quad q\in(1,2),
  \end{equation}
  where \(\hat S_a\) is introduced in (\ref{eq:defSs}c).
\end{enumerate}
\end{proposition}
 
\begin{proof} \textbf{Step~1.} We have by~\eqref{eq:Lh-psi},
\begin{equation}\label{eq:def-r-HK}
  r_{h,k}
  =
  \Bigl(h^2(\Delta\chi_k) - 2\ii(\nabla\chi_k)\cdot(-\ii h\nabla-\Ab)+\sum_{i\not=k}\mathfrak v_0(x-z_i) \Bigr) \psi_{h,k}
\end{equation}
where \(\chi_k(x)=\chi(x-z_k)\). 

Since \(r_{h,k}\) is supported in \(D(z_k,L+\eta)\setminus D(z_k,L-a)\), we get by using~\eqref{eq:d-main0} and the decay of the ground state \(\mathfrak u_h\) that
\begin{equation}\label{eq:est-r-HK}
  \norm{r_{h,k}}
  =
  \mathcal O(h^{-1/2}\ee^{-\hat S_a/h}).
\end{equation}

\textbf{Step 2.} We have
\[
  \norm{u_{h,k}}^2
  =
  \norm{\mathfrak u_h}^2-\int_{\R^2}(1-\chi^2)|\mathfrak u_{h}|^2\dd x.
\]
Since \(1-\chi^2\) is supported in \(\R^2\setminus D(0,L)\), we get by the decay of \(\mathfrak u_h\) that
\[
  \norm{u_{h,k}}^2
  =
  1 + \mathcal O(h^{-1}\ee^{-2S_0/h} )
\]
where \(S_0\) is introduced in (\ref{eq:defSs}a).

Let us now consider \(i\not=j\). We first inspect the case where  \(|z_i-z_j|=L\),  which  occurs only if \(j=i\pm 1\).  By a change of variable, we check that (thanks to the invariance by rotation)
\[ 
  \langle u_{h,i},u_{h,j}\rangle
  = 
  \begin{cases}
    \langle u_{h,1}, u_{h,2}\rangle 
    &\text{if } j = i + 1\\
    \overline{\langle u_{h,1}, u_{h,2}\rangle}
    &\text{if } j = i - 1.
  \end{cases}
\]
We perform the following decomposition 
\[ 
  \langle u_{h,1}, u_{h,2}\rangle
  =
  \mathcal E_1+\mathcal E_2
\]
where, 
\[
  \begin{aligned}
    \mathcal E_1
    &\coloneqq 
    \int_{D(z_1,L)}u_{h,1}(x)\overline{u_{h,2}(x)}\dd x,\\
    \mathcal E_2
    &\coloneqq 
    \int_{\substack{L<|x-z_1|<L+\eta\\ L<|x-z_2|<L+\eta}}u_{h,1}(x)\overline{u_{h,2}(x)}\dd x= \mathcal O(h^{-1/2}\ee^{- S_0/h}).
  \end{aligned}
\]
In \(D(z_1,L)\), the cut-off functions in the definitions of \(u_{h,1}\) and \(u_{h,2}\) are equal to \(1\). By a change of variable, we get
\[
  |\mathcal E_1|
  =
  \biggl\lvert\int_{D(0,L)} \mathfrak u_h(x)\mathfrak u_h(x_1+L,x_2)\ee^{\frac{\ii Lx_2}{2h}}\dd x\biggr\rvert,
\]
and by Lemma~\ref{lem:HK-int}, we get
\[
  \mathcal E_1=\mathcal O(h^{-1/2}\ee^{ -( S_a+o(1)) /h} )+\mathcal O(h^{-1/2}\ee^{-S_0 /h} )=\mathcal O(\ee^{-\hat S_a/h}),
\]
where in the last step we used the  inequalities \(\hat S_a<S_a<S_0\) from Proposition~\ref{prop:comp-SaS0}.  

If \(|z_i-z_j|\not=L\), then  \( \langle u_{h,i}\,,\,u_{h,j}\rangle=\mathcal O(h^{-1/2}\ee^{ -( S_a+o(1)) /h} )+\mathcal O(h^{-1/2}\ee^{-S_0 /h} )\) by a similar argument. 
 
\textbf{Step 3.} Let us now estimate \(\lambda_1(h)\) from below. Consider a partition of unity on $\R^2$, $\sum_{k=0}^{n}\zeta_k^2=1$, where
\[
  \text{for }
  1\leq k\leq n,\quad 
  \zeta_k=1 \text{ on } D\left(z_k,\frac{L}2\right),\quad  
  \supp\zeta_k\subset D(z_k,L-a)\,, 
\]
and
\[
  \supp \zeta_0\subset \Omega_0\coloneqq \R^2\setminus \bigcup_{k=1}^nD\left(z_k,\frac{L}2\right).
\] 
Pick a normalized ground state $f_h$   of  $\mathcal L_h$. Then we have,
\[ 
  \mathfrak v_0(x-z_i)\zeta_k(x)=0 \text{ for } i\neq k\,,~1\leq k\leq n\,,
\]
and
\[
  \mathfrak v_0(x-z_i)\zeta_0(x)=0 \text{ for } 1\leq i\leq n\,.
\]
We have the decomposition formula
\[
  \begin{multlined}
    \lambda_1(h)
    =
    \langle T_hf_h,f_h\rangle
    = 
    \sum_{k=1}^n \langle \mathcal L_h^{\rm sw} (\zeta_kf_h) ,\zeta_kf_h\rangle \\
    +\langle(-\ii h\nabla-\Ab)^2(\zeta_0f_h),\zeta_0f_h\rangle
    -h^2\sum_{k=0}^{n}\||\nabla\zeta_k|f_h \|^2
    -\lambda(h).
  \end{multlined}
\]
By the min-max principle we have
\[
  \sum_{k=1}^n \langle \mathcal L_h^{\rm sw} (\zeta_kf_h),\zeta_kf_h\rangle
  \geq 
  \lambda(h)\|\zeta_k f_h\|^2
\]
and (with~\eqref{eq:3.11a} in mind)
\[
  \langle(-\ii\nabla-\Ab)^2(\zeta_0f_h),\zeta_0f_h\rangle
  \geq 
  h\|\zeta_0 f_h\|^2
  \geq
  \lambda(h)\|\zeta_0f_h\|^2.
\]
Hence there exists $h_0>0$ such that for $h\in (0,h_0]$
\[
  \lambda_1(h)
  \geq 
  -M_0 n h^2
\]
where
\[
  M_0
  =
  \max_{0\leq k\leq n}\|\nabla\zeta_k\|_\infty^2.\qedhere
\]
\end{proof}

The next proposition allows us to verify Assumption~\ref{ass-error}. 

\begin{proposition}\label{prop:HK-sym}
With  \(J_0(h)=\langle T_hu_{h,1},u_{h,1}\rangle\) and \(J_1(h)=\langle T_hu_{h,1},u_{h,2}\rangle\)  we  have
\[
  J_0(h)=\mathcal O(\ee^{-2\hat S_a(v_0,L)/h})
\]
and
\[
  J_1(h)
  =
  \ee^{\ii\phi_n/2h} \mathfrak c_h(v_0,L) + \mathcal O(\ee^{-2\hat S_a(v_0,L)/h} ),
\]
where \(\mathfrak c_h( v_0,L)\) is the hopping coefficient introduced in~\eqref{eq:hopping},  \(\hat S_a(v_0,L)\) is introduced in  (\ref{eq:defSs}c) and
\[
  \phi_n=-\frac{L^2\sin(2\pi/n)}{2-2\cos(2\pi/n)}.
\]
Moreover, if \(2\hat S_a(v_0,L)>S(v_0,L)\) then we have
\[
  h\ln |J_1(h)|\underset{h\searrow0}{\sim}- S( v_0,L),
\]
where \(S(v_0,L)\) is introduced in~\eqref{eq:def-S(v0)**}.
\end{proposition}

\begin{proof}
From~\eqref{eq:def-r-HK}, we have
\[
  \begin{aligned}
    J_0(h)
    =
    \langle (\mathcal L_h-\lambda(h)) u_{h,1},u_{h,1}\rangle
    & = \int_{L-a\leq  |x-z_1|\leq L+\eta} r_{h,1}(x)\overline{u_{h,1}(x) }\dd x \\
    & = \mathcal O(h^{-1/2}\ee^{-2\hat S_a/h}).
  \end{aligned}
\]
We now move to estimate \(J_1(h)\).  First let us recall that the symmetry relations in Assumption~\ref{ass:symbis} (3) imply that \(u_{h,1}=M(g)u_{h,n}\) and\footnote{This formulation will be helpful since \(z_n\) lies on the \(x\)-axis.}
\[
  J_1(h)
  =
  \langle T_hu_{h,1},u_{h,2}\rangle
  =
  \langle T_hu_{h,n},u_{h,1}\rangle
  =
  \overline{\langle T_hu_{h,1},u_{h,n}\rangle}.
\]
Similarly to the previous  estimate of \(J_0(h)\),   by~\eqref{eq:Lh-psi} and~\eqref{eq:def-r-HK} we have
\[
  J_1(h)
  =
  J^{\rm app}_{1}(h)+\mathcal O(\ee^{-2\hat S_a/h} ),
\]
where
\[
  J^{\rm app}_{1}(h)
  =
  \int_{D(z_n,a) } \mathfrak v_0(x-z_n)\mathfrak u_h(x-z_1)\mathfrak u_h(x-z_n)\ee^{-\ii (z_n-z_1)\cdot\Ab(x)/h}\dd x.
\]
Notice that \((z_n-z_1)\cdot\Ab(x)=(z_n-z_1)\cdot\Ab(x-z_n+z_1)\). By a translation, we get
\[
  J^{\rm app}_{1}(h)
  =
  \ee^{-\ii(z_n-z_1)\cdot\Ab(z_1)/h}\int_{D(0,a) } \mathfrak v_0(y)\mathfrak u_h(y+z_n-z_1)\mathfrak u_h(y)\ee^{-\ii (z_n-z_1)\cdot\Ab(y)/h}\dd y.
\] 
With \(\ell_n=\frac{L}{\sqrt{2-2\cos(2\pi/n)}}\), we have \(z_n=(\ell_n,0)\) and \(z_1=\big(\ell_n\cos(2\pi/n),\ell_n\sin(2\pi/n)\big)\). Hence
\[
  (z_n-z_1)\cdot\Ab(z_1)=z_n\cdot\Ab(z_1)=\frac{\phi_n}{2}.
\]
Observing that \(|z_n-z_1|=L\), we write \(z_n-z_1=L\ee^{\ii\beta_n}\) where \(\beta_n\in(0,\pi]\). We denote by \(R_{\beta_n}\) the rotation by \(\beta_n\) around the origin. Noticing that
\[
  u\cdot\Ab(R_{\beta_n}v)=(R_{\beta_n}^{-1}u)\cdot \Ab(u)\quad (u,v\in\R^2),
\]
we get that 
\[
  (z_n-z_1)\cdot\Ab(R_{\beta_n} x)
  =
  \begin{pmatrix}
    L\\
    0
  \end{pmatrix}
  \cdot
  \begin{pmatrix}
    -x_2/2\\
    x_1/2
  \end{pmatrix}
  =
  -\frac{Lx_2}{2}.
\]
The change of variable \(y\mapsto x=R_{\beta_n}^{-1}y\) yields
\[
  J^{\rm app}_{1}(h)
  =
  \ee^{\ii\phi_n/2h}\int_{D(0,a) } \mathfrak v_0(x)\mathfrak u_h(x_1+L,x_2)\mathfrak u_h(x)\ee^{\ii Lx_2/h}\dd y
\]
where we used that the functions \(\mathfrak v_0,\mathfrak u_h\) are radial and that
\[
  |R_{\beta_n}x+z_n-z_1|=|x+R_{\beta_n}^{-1}(z_n-z_1)|=|(x_1+L,x_2)|.
\]
Now we have that  \(\ee^{-\ii\phi_n/2h}J^{\rm app}_{1}(h)=\mathfrak c_h(v_0,L)\) and by Proposition~\ref{prop:HK22-hopping}
\[
  h\ln|J_1^{\rm app}(h)|
  \underset{h\searrow0}{\sim}- S(  v_0,L),
\]
with \(0<S(v_0,L)<\hat S\). The same asymptotics hold for \(|J_1(h)|\) if we assume that \(2\hat S_a>S(v_0,L)\,\).
\end{proof}

From now on,  we work under the following assumption on \((v_0,L)\).

\begin{assumption}\label{ass:v0L}
With \(a=a(v_0)\)   introduced  introduced in~\eqref{eq:def-a} and \(L>2a\),  the function \(v_0\) satisfies \(2\hat S_a(v_0,L)> S(v_0,L)\), where \(\hat S_a(v_0,L)\) and \(S(v_0,L)\) are introduced in~\eqref{eq:defSs} and~\eqref{eq:def-S(v0)**} respectively.
\end{assumption}

\begin{remark}\label{rem:ass-v0L}
By Proposition~\ref{prop:comp-SaS0}, Assumption~\ref{ass:v0L} holds if \(L>4a(v_0)\).
\end{remark}

We fix the choice of \(\mathfrak S_1,\mathfrak S_2,\mathfrak S_3\) as in~\eqref{eq:condhatS} but with the additional condition that
\[
  S(v_0,L)<2\min_{1\leq j\leq 3}\mathfrak S_j.
\]
This is possible under Assumption~\ref{ass:v0L} since  \(S(v_0,L) < 2 \hat S_a(v_0,L)\).  Therefore,~\eqref{eq:newcond} holds with \(S=S(v_0,L)\) and Proposition~\ref{prop:HK-sym} ensures that the other conditions in Assumption~\ref{ass-error} hold too\footnote{This is explicitly done for \(n=2,3\) and can be easily justified for \(n\geq 4\).}.

\subsection{The case $n=2$: Double wells}
We choose here \(V\) to be the \emph{double well}  potential defined by~\eqref{eq:V} for \(n=2\). We therefore have two wells
\begin{equation}\label{eq:z-lr}
  z_1
  =
  \Big(-\frac{L}2,0\Big),\quad z_2=\Big(\frac{L}2,0\Big),\quad L>2a>0.
\end{equation}
A straightforward application of~\eqref{eq:tun-error-a},  Propositions~\ref{prop:HK-ass1}  and \ref{prop:HK-sym} (and Remark~\ref{rem:optimal-error}) yields the following.
\begin{theorem}\label{thm:HK2wells}
Assuming  the conditions in~\eqref{eq:v0} and in Assumption~\ref{ass:v0L} are fulfilled,  then  the following asymptotics holds,
\begin{equation}\label{eq:HK2wells}
  h \ln\big(\lambda_2(h)-\lambda_1(h)\big) \underset{h\searrow0}{\sim} -  S(v_0,L),
\end{equation}
where \(S(v_0,L)\) is the constant introduced in~\eqref{eq:def-S(v0)**}.
\end{theorem}

\subsubsection{Discussion}

Let us introduce the following classes\footnote{\(\mathscr A^{\rm FSW}\) is the admissible class introduced by Fefferman-Shapiro-Weinstein and HKS refers to the admissible classes introduced in this paper.}  of admissible \((v_0,L)\),  where \(v_0\) satisfies the conditions in \eqref{eq:v0},  and
\[
  \begin{aligned}
    \mathscr A
    &=\{(v_0,L)~:~L>2a(v_0) \mbox{ and \eqref{eq:HK2wells} holds}\},\\
    \mathscr A^{\rm FSW}
    &=\{(v_0,L)~:~L>4(a(v_0)+\sqrt{|v_0^{\min}|})\},\\
    \overline{\mathscr A}^{\rm HKS}
    &=\{(v_0,L)~:~v_0\mbox{ satisfies Assumption \ref{ass:v0L}}\}\\
    \underline{\mathscr A}^{\rm HKS}
    &=\{(v_0,L)~:~L\geq 4a(v_0)\}.
  \end{aligned}
\]
By Proposition~\ref{prop:comp-SaS0},  we have
\[
  \underline{\mathscr A}^{\rm HKS}\subset \overline{\mathscr A}^{\rm HKS}
\]
and by Theorem~\ref{thm:HK2wells}
\[
  \overline{\mathscr A}^{\rm HKS}\subset \mathscr A.
\]
The earlier results in~\cite{FSW, HK22} yield that
\[
  \mathscr A^{\rm FSW}\subset \mathscr A.
\]
However,  our results are much stronger since 
\(\mathscr A^{\rm FSW}\) is a proper subset of \(\underline{\mathscr A}^{\rm HKS}\).

\subsubsection{Proof of Theorem~\ref{thm:HKS1}}

In light of~\eqref{eq:red-b=1}, it suffices to apply Theorem~\ref{thm:HK2wells} with the effective semi-classical parameter \(\hbar =b^{-1}h\) and the effective potential defined by \(b^{-2}v_0\).  We then obtain \(\mathscr E_{b,L}(v_0)=bS(b^{-2}v_0,L)\).  

\subsection{The case \(n=3\): Three wells and flux effects}

Let us assume that \(n=3\) so that \(V\) defined by~\eqref{eq:V} is a potential with three wells 
\[
  z_1
  =
  \frac{L}{2\sqrt{3}}(-1,\sqrt{3}), \quad z_2=\frac{L}{2\sqrt{3}}(-1,-\sqrt{3}),\quad z_3= \frac{L}{\sqrt{3}}(1,0),
\]
which are located on the vertices of an equilateral triangle  with side length \(L\) and let \(D\subset\R^2\) be its interior.  The area of \(D\) is then
\(\frac{\sqrt{3}}{4}L^2\).  The flux of the magnetic field in  \(D\) is
\begin{equation}\label{eq:def-flux-n=3}
  \Phi
  \coloneqq
  \frac1{2\pi}\int_D\curl\Ab\dd x
  =
  \frac{\sqrt{3}L^2}{8\pi}.
\end{equation}
By Proposition~\ref{prop:HK-sym}, we have
\[
  J_1(h)
  \underset{h\searrow0}{\sim} 
  |J_1(h)|\ee^{-2\pi \ii \Phi/3h}, \quad 
  |J_1(h)|
  \underset{h\searrow0}{=}
  \exp\left(-\frac{ S(\mathfrak v_0,L)+o(1)}{h} \right),
\]
and
\[
  J_0(h)
  =
  \mathcal O(\ee^{-2\hat S_a/h}) \mbox{ with } S(\mathfrak v_0,L)<2\hat S_a\,. 
\]

Recall that  by~\eqref{eq:Ik=Jk},
\begin{equation}\label{eq:I1(h)=J1(h)*}
  I_1(h)
  \underset{h\searrow0}{\sim}
  J_1(h).
\end{equation}
If for a given constant \(c_0>0\) we introduce the set
\begin{equation}\label{eq:hyp-sep}
  \mathcal N(c_0)
  =
  \left\{h\in(0,1],~\dist\left(\frac{\Phi}{3h},\Z\right)\geq c_0\right\},
 \end{equation}
then it results from~\eqref{eq:I1(h)=J1(h)*} that (since \(\theta(h)\in[0,2\pi)\) in the definition of \(I_1(h)\))
\begin{equation}\label{eq:app-theta}
  \theta(h)
  \underset{\substack  {h\searrow0\\
  h\in\mathcal N(c_0)}}{=}
  2\pi\left(\frac{\Phi}{3h}-\left\lfloor\frac{\Phi}{3h}\right\rfloor \right)+o(1). 
\end{equation}
Now  by  paragraph~\ref{sec:3-wells}  we get crossing of eigenvalues  quantified via the following functions
\begin{equation}\label{eq:def-a-b}
  \begin{aligned}
    \mathsf a(x)
    &=\cos\left(\frac{x}{3}+\frac{2\pi}{3}\right)-\cos\left(\frac{x}{3}\right),\\
    \mathsf b(x)
    &=\cos\left(\frac{x}{3}+\frac{4\pi}{3}\right)-\cos\left(\frac{x}{3}+\frac{2\pi}{3}\right).
  \end{aligned}
\end{equation}

\begin{theorem}\label{thm:HK-3wells}
Assume that  the conditions in~\eqref{eq:v0} are fulfilled and that \(V\) is defined by~\eqref{eq:V} with \(n=3\).  If Assumption~\ref{ass:v0L} holds,  then  there is a relabeling  \(\mu_1(h),\mu_2(h),\mu_3(h)\) of the eigenvalues \(\lambda_1(h),\lambda_2(h),\lambda_3(h)\) of the electro-magnetic Schr\"odinger operator \(\mathcal L_h\)  such that the asymptotics 
\begin{equation}\label{eq:HK3wells}
  \begin{aligned}
    \mu_2(h)-\mu_1(h)
    &\underset{h\searrow0}{=} \bigl(\mathsf a(\Phi/h)+o(1)\bigr) \exp\left(\frac{- S( v_0,L)+o(1)}{h}\right)\\
    \mu_3(h)-\mu_2(h) 
    &\underset{h\searrow0}{=}\bigl(\mathsf b(\Phi/h) +o(1)\bigr) \exp\left(\frac{- S( v_0,L)+o(1)}{h}\right)
  \end{aligned}
\end{equation}
hold  for all  \(L>2a(v_0)\).
 
Moreover,  there exists a sequence \(\bigl(h_1(k),h_2(k),h_3(k)\bigr)_{k\geq k_0}\) which converges to \(0\) such that,  for all \(k\geq k_0\) we have
\[
  0<h_1(k+1)<h_3(k)<h_2(k)<h_1(k)<1
\]
and
\[ 
  \mu_1\big(h_1(k)\big)=\mu_3\big(h_1(k)\big),~  \mu_1\big(h_2(k)\big)=\mu_2\big(h_2(k)\big),~ \mu_2\big(h_3(k)\big)=\mu_3\big(h_3(k)\big).
\]
\end{theorem}

\begin{proof}
To get the asymptotics in~\eqref{eq:HK3wells},  we use the computations in  Paragraph~\ref{sec:3-wells} and~\eqref{eq:I1(h)=J1(h)*}.  Then we approximate \(|J_1(h)|\) by  Proposition~\ref{prop:HK-sym}. 

Consider constants \(\delta_i\), \(i=0,1,2,3\),  such that
\[
  0<\delta_0<\frac16<\delta_1<\frac13<\delta_2<\frac12<\delta_3<1.
\]
For all \(k\geq 1\),  consider 
\[
  h_0^*(k)>h_1^*(k)>h_2^*(k)>h_3^*(k)
\]
defined by
\[
  \frac{\Phi}{3h_i^*(k)}=k+\delta_i \quad (0\leq i\leq 3).
\]
Notice that \(h_3^*(k)>h_0^*(k+1)\) and by~\eqref{eq:app-theta}, there exists \(k_0\geq 1\) such that, for all \(k\geq k_0\), we have
\[
  0<\theta\bigl(h_0^*(k)\bigr)<\frac\pi3<\theta\bigl(h_1^*(k)\bigr)<\frac{2\pi}{3}<\theta\bigl(h_2^*(k)\bigr)<\pi<\theta\bigl(h_3^*(k)\bigr)<2\pi.
\]
By continuity,  we select 
\[
  h_1(k)\in \bigl(h_0^*(k),h_1^*(k)\bigr),\quad h_2(k)\in \bigl(h_1^*(k),h_2^*(k)\bigr),\quad h_3\in \bigl(h_2^*(k),h_3^*(k)\bigr),
\]
such that
\[
  \theta\bigl(h_1(k)\bigr)=\frac\pi3,\quad \theta\bigl(h_2(k)\bigr)=\frac{2\pi}{3},\quad \theta_3\bigl(h_3(k)\bigr)=\pi
\]
and therefore we get the eigenvalue crossings as indicated in Paragraph~\ref{sec:3-wells}.
\end{proof} 

\section{Smoothed triangles and Neumann boundary condition}\label{sec.Neumann}

\subsection{Introduction}
\subsubsection{Geometric setting}\label{sec:geom}
In this section, $\Omega$ is a bounded open set of $\mathbb R^2$ with $C^\infty$ boundary \(\Gamma\).  The Hilbert space is $H=L^2(\Omega)$ and we will assume (see below for the invariance assumption)  that we are in the situation considered  in Remark \ref{remlessabstract}.  

We assume that \(\Gamma\) is a simple curve and denote its length by \(|\Gamma|=2L\). Let \(\R/2L\Z\ni s\mapsto \gamma(s)\) be the arc-length parameterization of \(\Gamma\) such that the unit tangent vector \(\tau(s)\coloneqq \dot\gamma(s)\) turns counterclockwise.  Let us denote by \(k\)  the curvature along \(\Gamma\) defined as follows
\[
  \ddot\gamma(s)=k(s)\nu(s)
\]
where \(\nu(s)\) is the unit normal to \(\Gamma\) at \(\gamma(s)\) pointing  inward  \(\Omega\).

\begin{definition}\label{def:curv-well}
We introduce the maximal curvature along \(\Gamma\) as follows
\begin{equation}\label{eq:def-curv}
k_{\max}=\max_{s\in \R/2L\Z} k(s)
\end{equation}
and call a point \(z_0\in\Gamma\) a \emph{curvature well} if \(z_0=\gamma(s_0)\) and \(k(s_0)=k_{\max}\). The point \(z_0\) is said to be a non-degenerate curvature well if furthermore \(k''(s_0)<0\).  We denote by \(\Gamma_0\) the set of curvature wells. 
\end{definition}

Consider a positive integer \(n\) and the rotation  by \(2\pi/n\) denoted by \(g_n\).   
We assume that
\begin{subequations}\label{eq:4.1}
\begin{equation}\label{eq:4.1a}
  \Omega \text{ is invariant by the rotation } g_n
\end{equation}
and  that we have \(n\) non-degenerate curvature wells
\begin{equation}\label{eq:4.1b}
  \Gamma_0=\{z_1=\gamma(s_1),\cdots,z_n=\gamma(s_n)\}
\end{equation}
with\footnote{ We identify \(\R/2L\Z\) and \([0,2L)\).}
\begin{equation}\label{eq:4.1c}
  s_j=(j-1)\frac{2L}{n}, \quad k''(s_j)<0\quad (1\leq j\leq n).
\end{equation}

The symmetry assumption yields that
\begin{equation}\label{eq:zk+1=zk}
  z_{j+1}=g_nz_j\quad(1\leq j\leq n-1).
\end{equation}
and
\begin{equation}\label{eq:gamma(s+s1)}
  \gamma\left(s+\frac{2L}{n}\right)
  =
  g_n\gamma(s),\quad k\left(s+\frac{2L}{n}\right)=k(s)\quad (s\in\R/2\pi\Z).
\end{equation}
\end{subequations}

\subsubsection{The magnetic Neumann Laplacian}\label{sec:Neumann}

We are interested in the magnetic Neumann Laplacian \(\mathscr L_h^N=(-\ii h\nabla-\Ab)^2\),  in the Hilbert space  \(L^2(\Omega)\),  where the magnetic field \(B=\curl\Ab=1\) is uniform\footnote{We can handle any magnetic field intensity \(b>0\) by a change of semi-classical parameter.} and generated by the magnetic potential \(\Ab\) introduced in~\eqref{eq:A(x)}. The operator \(\mathscr L_h^N\) acts on functions \(u\in H^2(\Omega)\) satisfying the (magnetic) Neumann condition
\[
  \nu\cdot (-\ii h\nabla-\Ab)u|_{\Gamma}=0.
\]
The case of double curvature wells corresponding to $n=2$ was treated in~\cite{BHR},  where the symmetry was generated by the reflection $\tilde g_2$ (see Remark \ref{rem:reflection}).  Here we focus on the case with \(n\geq 3\) curvature wells and where the symmetry is generated by 
a rotation.  When \(n=3\),  a typical example is the smoothed triangle (Fig.~\ref{fig1-intro}). When $\Omega$ is an equilateral triangle the heuristic discussion is given in~\cite{FH-b} but no rigorous result can be given since the authors were unable to have a sufficiently accurate control of the tunneling.  Numerically,  this has been computed  in~\cite{BDMV}
 which in particular gives the enlightening picture predicting eigenvalue crossings (Fig.~\ref{VPEquicopy}).

 \begin{figure}[htb]
  \centering
  \includegraphics[scale=0.4]{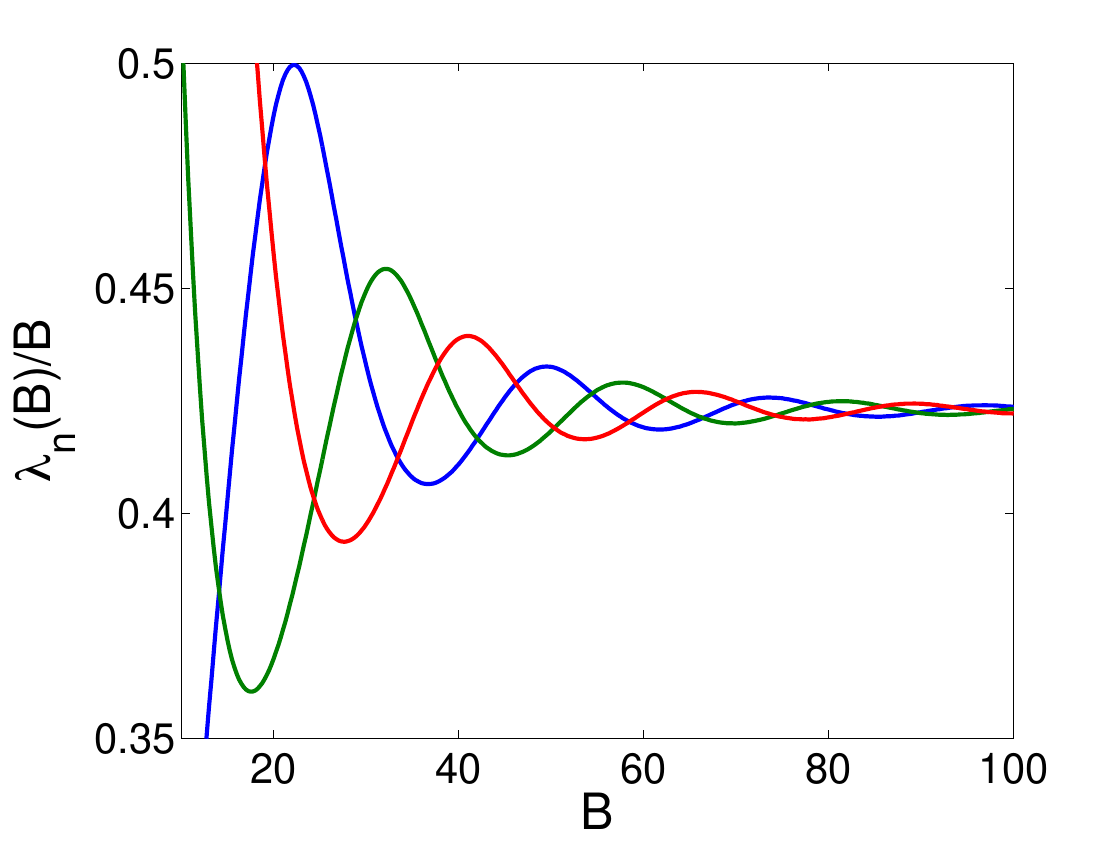} 
  \caption{Eigenvalue crossings in an equilateral triangle.}\label{VPEquicopy}
\end{figure}

The investigation of \(\mathscr L_h^N\) can be connected with Section~\ref{sec.IM}  after  shifting by a constant \(\ell(h)\) and taking the operator \(T_h\) as follows
\begin{equation}\label{eq:N-def-Th}
  T_h
  =
  \mathscr L_h^N-\ell(h).
\end{equation}
The shift constant \(\ell(h)\) will be defined by the ground state energy of  an operator with a single curvature well.  
 
Two  constants are important in our analysis.  First we meet  a magnetic flux like term  that controls the eigenvalue crossings and is defined by
\begin{subequations}\label{eq:N-constants}
\begin{equation}\label{eq:def-flux-Om}
  \Phi_0
  =
  |\Omega|-|\Gamma|\xi_0=|\Omega|-2L\xi_0,
\end{equation}
Secondly  the following constant controls the strength of the tunneling and is defined by
\begin{equation}\label{eq:def-S-N}
  \mathsf S_n=\sqrt{ \frac{2C_1}{(\mu_1^N)''(\xi_0)}}\int_0^{2L/n} \sqrt{k_{\max}-k(s)}\dd s.
\end{equation}
\end{subequations}
The definition of  \(\Phi_0\) and \(\mathsf S_n\)  involves universal constants related to the de\,Gennes model. We recall that, for \(\xi\in\R\),   \(\mu_1^N(\xi)\) denotes  the lowest eigenvalue of the Neumann realization of the Harmonic oscillator on the semi-axis  $\mathbb R_+$:
\begin{subequations}\label{eq:dG-model}
\begin{equation}\label{eq:dG-model-a}
  \mathfrak h^{N}[\xi]=-\frac{d^2}{d\tau^2}+(\xi+\tau)^2
\end{equation}
Minimizing over \(\xi\in\R\) we get the de\,Gennes constant
\begin{equation}\label{eq:dG-model-b}
  \Theta_0
  =
  \inf_{\xi\in\mathbb R}\mu^{N}_1(\xi)=\mu^{N}_1(\xi_0),\mbox{ where } \xi_0=-\sqrt{\Theta_0}\,.
\end{equation}
Denoting by  $u_0$  the positive $L^2$- normalized ground state of $\mathfrak h^N[\xi_0]$,   the constant $C_1$ appearing in~\eqref{eq:def-S-N} is defined by
\begin{equation}\label{eq:dG-model-c}
  C_1
  =
  \frac{|u_0(0)|^2}{3}\,.
\end{equation}
\end{subequations}

\subsection{Reduction to tubular domain}
In a tubular neighborhood of the boundary,  we can work with adapted coordinates \((s,t)\in \R/2L\Z\times \R_+\) that are linked to Cartesian coordinates as follows
\begin{equation}\label{eq:def-transf}
  (x_1,x_2)=\mathscr T(s,t)\coloneqq \gamma(s)-t\nu(s).
\end{equation}
There exists a geometric constant \(\varepsilon_0>0\) such that  the above transformation is invertible when \(0<t<\varepsilon_0\); the image of \(\Phi\) is the tubular neighborhood of \(\Gamma\)
\begin{equation}\label{eq:def-Gam(epsilon)}
  \Gamma(\varepsilon_0)
  =
  \{x\in\Omega~:~{\rm dist}(x,\partial\Omega)<\varepsilon_0\} 
\end{equation}
and for \(x\in\Gamma(\varepsilon_0)\),  \((s,t)=\mathscr T^{-1}(x)\) means that \(s\) is the arc-length coordinate of the projection of \(x\) on \(\Gamma\) and \(t\) is the normal distance from \(x\) to \(\Gamma\).  We will refer to \(s\) as the tangential variable and to \(t\) as the normal variable.

Let us understand the action of the operator in~\eqref{eq:def-Mgn} in these adapted coordinates.  Let \(u\) be supported in \(\Gamma(\varepsilon_0)\) and \(v=M(g)u\).  Then by~\eqref{eq:gamma(s+s1)},  \(\tilde v(s,t)=v\circ\mathscr T\) satisfies
\begin{equation}\label{eq:def-Mgn-bnd}
  \tilde v(s,t)
  =
  \tilde u\left(s-\frac{2L}{n},t \right),\quad \tilde u=u\circ\mathscr T.
\end{equation}  
The Hilbert space \(L^2(\Gamma(\varepsilon_0))\) is transformed to the weighted space  
\[
  L^2\big((\mathbb R/ 2L \mathbb Z) \times(0,\varepsilon_0);a\dd s \dd t\big),\quad a(s,t)=1-tk(s).
\]
After a gauge transformation to eliminate the normal component of $\mathbf A$, see~\cite[App.~F]{FH-b}, the action of the  operator $\mathscr L_h^N$ is  transformed into  
\[
  \begin{multlined}
  \tilde{\mathscr L}_h\coloneqq -h^2 a^{-1}\partial_t  a\partial_t + a^{-1}\left(-\ii h\partial_s+\gamma_0-t+\frac{k(s)}{2}t^2\right) \\
  \times a^{-1}\left(-\ii h\partial_s+\gamma_0-t+\frac{k(s)}{2}t^2\right)
  \end{multlined}
\]
where 
\begin{equation}\label{eq:gamma0}
\gamma_0=\frac{|\Omega|}{2L}=\frac{\pi}{L} \Phi
\end{equation}
and \[
\Phi=\frac1{2\pi}\int_\Omega B\dd x
\mbox{ is the magnetic flux  in } \Omega\,.
\]
   The   change of variables, $t=h^{1/2}\tau$ and $s=\sigma$,  transforms \(\mathbb R/ 2L \mathbb Z\times(0,\varepsilon_0)\) and the measure to \(a\dd s\dd t\) to
   \[\tilde\Gamma_h=(\mathbb R/ 2L \mathbb Z)\,\times(0,\varepsilon_0 h^{-1/2})\quad{\rm and }\quad \tilde a_h(\sigma,\tau)=1-h^{1/2}\tau k(\sigma).\]
Moreover,  it transforms the    Hilbert space
\(L^2\big((\mathbb R/ 2L \mathbb Z)\,\times(0,\varepsilon_0);a\dd s\dd t\big)\) to the Hilbert space
\(L^2\big(\tilde\Gamma_h;\tilde a_h\dd\sigma \dd\tau\big),\)
and the operator \(\tilde{\mathscr L}_h\) to \(h\tilde{\mathscr N}_h\) where
\begin{multline}\label{eq:tilde-Nh}
  \tilde{\mathscr N}_{h}
  =
  -\tilde a^{-1}_{h}\partial_\tau \tilde a_{h}\partial_\tau
  +\tilde{a}^{-1}_{h}\left(-\ii h^{1/2}\partial_\sigma+{ h^{-1/2}\gamma_0}-\tau +h^{1/2} \frac{k(\sigma)}{2}\tau^2\right)\\
  \times \tilde{a}^{-1}_{h}\left(-\ii h^{1/2}\partial_\sigma+{ h^{-1/2}\gamma_0}-\tau+h^{1/2} \frac{k(\sigma)}{2}\tau^2\right).
\end{multline}
Let us consider \(\tilde{\mathscr N}_h\), in \(L^2\big(\tilde\Gamma_h;\tilde a_h\dd\sigma d\tau\big)\), with domain\footnote{ Since \(\tilde a_h=\mathcal O(\varepsilon_0\|k\|_\infty)\), and \(\varepsilon_0\ll 1\),  the  vector space with weighted measure \(L^2\big(\tilde\Gamma_h;\tilde a_h\dd\sigma d\tau\big)\) is the same as the space  with the flat measure \(L^2\big(\tilde\Gamma_h;d\sigma d\tau\big)\) with equivalent norm.}
\[
  \begin{aligned}
    \mathsf{Dom}(\tilde{\mathscr N}_h)
    =
    \{u\in L^2\big(\tilde\Gamma_h;d\sigma d\tau\big)~:~&\partial_\tau^2u,\partial_\sigma^2u\in L^2\big(\tilde\Gamma_h;d\sigma d\tau\big),\\
    &
    \partial_\tau u|_{\tau=0}=0,~u|_{\tau=\varepsilon_0h^{-1/2}}=0\}.
  \end{aligned}
\]
The study of the eigenvalues of \(\mathscr L_h^N\),  can then be compared with those of \(\tilde{\mathscr N}_h\) (see~\cite[Prop.~2.7]{BHR}).  

\begin{proposition}\label{prop:BHR2.7}
  Let \(N\in\N\) and $\mathsf S_n$ be   the constant introduced in~\eqref{eq:def-S-N}.  There exist $K>\mathsf{S}_n$, $C,h_0>0$ such that, for all $h\in(0,h_0]$ and \(1\leq k\leq N\), we have,
  \[
    \lambda_k (\mathscr L_h^N)-Ce^{-K/h^{\frac 14}}\,\leq  h\lambda_k(\tilde{\mathscr N}_h)\leq \lambda_k(\mathscr L_h^N)+Ce^{-K/h^{\frac 14}}\,,
  \]
  where $\lambda_k(\mathscr L_h^N)$ and $\lambda_k(\tilde{\mathcal N}_h)$ are the $k$-th eigenvalues, counting multiplicity,    of the operators $\mathscr L_h^N$ and $\tilde{\mathscr N}_{h}$, respectively.
\end{proposition}
Recall that we are interested in applying the results in Section~\ref{sec.IM} to the operator \(T_h\) obtained by shifting the operator \(\mathscr L_h^N\) (see~\eqref{eq:N-def-Th}). Effectively, that is related to the operator obtained by doing the corresponding shift to the operator  \(\tilde{\mathscr N}_h\),
\begin{equation}\label{eq:N-def-Th-t}
  \tilde T_h
  =
  \tilde{\mathscr N}_h-h^{-1}\ell(h).
\end{equation}
Our next task is to verify Assumptions~\ref{ass1} and \ref{ass:symbis} (they will both hold for \(\tilde T_h\) and \(T_h\)) and this will require the construction of certain quasi-modes \((u_{h,i})_{1\leq i\leq n}\).
Let us make the following two observations:

\begin{itemize}
  \item [(i)]
  The  assumption in (\ref{eq:4.1}a)  implies that $\tilde T_h$ commutes with $M(g_n)$,  hence the condition of invariance by rotation holds.
  \item[(ii)]
  For every fixed labeling \(n\),  the eigenfunctions of $\tilde T_h$ corresponding to the \(n\)'th eigenvalue decay exponentially in the (rescaled) tangential variable; more precisely,  given an eigenfunction \(f_n\) of \(\tilde T_h\) with corresponding eigenvalue \(\lambda_n(\tilde T_h\),  there exist positive constants \(C_n,h_n,\alpha_n\) such that 
  \[ 
    \forall\,h\in(0,h_n],\quad \int_{\tilde\Gamma_h}\ee^{\alpha_k\tau}\bigl(|\partial_\tau f_n|^2+|f_n|^2 \bigr)\dd s\dd\tau\leq C_k\int_{\tilde\Gamma_h}|f_n|^2\dd s\dd \tau.
  \]
\end{itemize}

The relevance of (i) above is that we can construct quasi-modes of \(\tilde T_h\) obeying the symmetry invariance properties as in Assumption~\ref{ass:symbis},  and in turn we can use these quasi-modes to produce quasi-modes for the initial operator \(T_h\) in~\eqref{eq:N-def-Th}. The observation in (ii) asserts that the eigenfunctions of \(\tilde T_h\),  once rescaled to the initial tangential variable \(t=h^{1/2}\tau\),  can be ignored in the interior of the domain \(\Omega\).  

The expression of the operator \(\tilde{\mathscr N}_h\) (and hence \(\tilde T_h\)) involves the effective semi-classical parameter \(\hbar\coloneqq h^{1/2}\). This leads us to adjust our setting by working with the  operators
\begin{equation}\label{eq:N-def-Th-bar}
  \mathscr T_\hbar
  =
  \mathscr N_\hbar-\lambda(\hbar)
\end{equation}
and
\begin{multline}\label{eq:tilde-Nh-bar}
  \mathscr N_{\hbar}=-a^{-1}_{\hbar}\partial_\tau  a_{\hbar}\partial_\tau
  +{a}^{-1}_{\hbar}\left(-\ii \hbar\partial_\sigma+{ \hbar^{-1}\gamma_0}-\tau +\hbar \frac{k(\sigma)}{2}\tau^2\right)\\
  \times {a}^{-1}_{\hbar}\left(-\ii \hbar\partial_\sigma+{ \hbar^{-1}\gamma_0}-\tau +\hbar \frac{k(\sigma)}{2}\tau^2\right)
\end{multline}
where
\[
  a_\hbar(\sigma,\tau)
  =
  1-\hbar \tau k(\sigma)
  =
  \tilde a_h(\sigma,\tau),
  \quad 
  \lambda(\hbar)
  =
  h^{-1}\ell(h)\,.
\]

\subsection{The single well problem}

\subsubsection{Definition of the operator}

Let us recall that by~\eqref{eq:4.1a} and~\eqref{eq:4.1b},  we have on the interval \( (-2L/n,2L/n)\) a single  non-degenerate  maximum,  \(s_1=0\), 
of the curvature \(k(s)\).  Let us fix  a positive \(\eta<\min(\frac14,\frac{L}{4n})\) and consider the following set
\[\tilde\Gamma_{\hbar,\eta}^{(1)}=\Bigl(-\frac{2L}{n}+\eta,\frac{2L}{n}-\eta\Bigr)\times (0,\varepsilon_0\hbar^{-1}).\]
Now we consider the operator
\begin{multline}\label{eq:tilde-Nh-bar-1}
  \mathscr N_{\hbar}^{(1)}=-a^{-1}_{\hbar}\partial_\tau  a_{\hbar}\partial_\tau +{a}^{-1}_{\hbar}\left(-\ii \hbar\partial_\sigma+{ \hbar^{-1}\gamma_0}-\tau +\hbar \frac{k(\sigma)}{2}\tau^2\right)\\
  \times {a}^{-1}_{\hbar}\left(-\ii \hbar\partial_\sigma+{ \hbar^{-1}\gamma_0}-\tau +\hbar \frac{k(\sigma)}{2}\tau^2\right)
\end{multline}
with domain 
\[
  \begin{aligned}
    \mathsf{Dom}(\mathscr N_{\hbar}^{(1)})=\{u\in L^2\big(\tilde\Gamma_{\hbar,\eta}^{(1)};d\sigma d\tau\big)~:~&\partial_\tau^2u,\partial_\sigma^2u\in L^2\big(\tilde\Gamma_{\hbar,\eta}^{(1)};d\sigma d\tau\big),\\
    &\partial_\tau u|_{\tau=0}=0,~u|_{\tau=\varepsilon_0h^{-1/2}}=0,~u|_{\sigma=\pm L/4n}=0\}.
  \end{aligned}
\]
We denote by \(\lambda(\hbar)\) the ground state energy of the operator \(\mathscr N_{\hbar}^{(1)}\); it is simple and can be expanded as follows~\cite{FH06,  BHR16}
\begin{subequations}\label{eq:N-WKB-ev}
\begin{equation}\label{eq:N-WKB-ev1} 
  \lambda(\hbar)\underset{\hbar\to0}{\sim}\Theta_0-3C_1k_{\max}\hbar+C_1\Theta_0^{1/4}\sqrt{\frac{3|k_2|}{2}}\hbar^{3/2}+\sum_{j\geq 4}\delta_{1,j}\hbar^{j/2}.
\end{equation}
where \((\delta_{1,j})_{j\geq 1}\) are real constants and \(k_2=k_2''(0)<0\). Moreover there exist real constants  \((\delta_{2,j})_{j\geq 1}\) such that the second eigenvalue satisfies
\begin{equation}\label{eq:N-WKB-ev2} 
  \lambda_2(\hbar)\underset{\hbar\to0}{\sim}\Theta_0-3C_1k_{\max}\hbar+3C_1\Theta_0^{1/4}\sqrt{\frac{3|k_2|}{2}}\hbar^{3/2}+\sum_{j\geq 4}\delta_{2,j}\hbar^{j/2}
\end{equation}
\end{subequations}
and there is a spectral  gap
\[
  \lambda_2(\hbar) -\lambda(\hbar) \underset{\hbar\to0}{\sim}  2 C_1\Theta_0^{1/4}\sqrt{\frac{3|k_2|}{2}}\hbar^{3/2}.
\]
Functions in the domain of \(\mathscr N_{\hbar}^{(1)}\) can be extended to the full half-plane \(\R_+^2=\R\times\R_+\) after  multiplication by a suitable cutoff function.  We could have considered the single well problem in an alternative manner by truncating  the curvature and extending it by \(0\) outside \(\hat\Gamma_{\hbar,\eta}^{(1)}\) and  also the weight function \(a_\hbar\) to get an operator in \(\R^2_+\).  Due to the exponential decay of bound states,  the spectra agree up to exponentially small errors that are negligible compared with the estimate of the tunneling~\cite[Prop.~2.7]{BHR}.

\subsubsection{Approximation of ground states}

The ground states of \(\mathscr N_{\hbar}^{(1)}\) concentrate near the curvature well \(s_1=0\)~\cite[Corol.6.1]{BHR}  and one can expand them in a WKB form.   

Let \(\phi_{\hbar,1}\)  be a normalized ground state of  \(\mathscr N_{\hbar}^{(1)}\).   We introduce the  function
\begin{equation}\label{eq:N-Agmon}
  \Phi_1(\sigma)
  =
  \sqrt{\frac{2C_1}{(\mu_1^N)''(\xi_0)}}\int_{[0,\sigma]}\sqrt{k_{\max}-k(\varsigma)}\dd \varsigma,
\end{equation}
where \([0,\sigma]\) is the segment joining \(0\) to \(\sigma\)  oriented counter-clockwise; in particular \(\int_{[0,\sigma]}=\int_{\sigma}^0\) if \(\sigma< 0\) and    \(\int_{[0,\sigma]}=\int_0^{\sigma}\) if \(\sigma> 0\).
The ground state \(\phi_{\hbar,1}\) is approximated as follows.  Let \(K\subset \bigl(-\frac{L}{2n}+\eta,\frac{L}{2n}-\eta\bigr)\) be a compact interval,  then we have~\cite[Prop.~6.3~\&~Eq.~(2.5)]{BHR}
\begin{subequations}\label{eq:N-WKB}
\begin{equation}\label{eq:N-WKB-a}
  \Bigl\|\langle\tau\rangle\ee^{\Phi_1/\sqrt{\hbar}}\bigl(\phi_{\hbar,1}-\ee^{-\ii\gamma_0\sigma/\hbar^2}\psi_{\hbar,1} \bigr)\Bigr\|_{C^1(K;L^2(\R_+))}=\mathcal O(h^\infty),
\end{equation}
where \(\langle \tau\rangle=(1+|\tau|^2)^{1/2}\) and \(\gamma_0\) is introduced in~\eqref{eq:gamma0}. \\  The function \(\psi_{\hbar,1}\) is defined as follows
\begin{equation}\label{eq:N-WKB-b}
  \psi_{\hbar,1}
  =
  \chi\Psi_{\hbar,1}
\end{equation}
where \(\chi\in C^{\infty}(\R;[0,1])\) satisfies 
\begin{equation}\label{eq:cond-chi-N}
  \chi
  =
  1\text{ on }\Bigl(-\frac{L}{2n}+2\eta,\frac{L}{2n}-2\eta\Bigr),~ {\rm supp}\chi\subset \Bigl(-\frac{L}{2n}+\eta,\frac{L}{2n}-\eta\Bigr)
\end{equation}
and the function \(\Psi_{\hbar,1}\) has the following expansion~\cite[Thm.~2.8]{BHR}
\begin{equation}\label{eq:N-WKB-c}
\ee^{\Phi_1(\sigma)/\hbar^{\frac 12}} \ee^{-\ii \sigma\xi_0/\hbar}\Psi_{\hbar,1}(\sigma,\tau)\underset{\hbar\to0}{\sim}\hbar^{-\frac18} \sum_{j\geq 0} b_{j}(\sigma,\tau) \hbar^{\frac j2},
\end{equation}
where 
\[
  b_{0}(\sigma,\tau)=f_0(\sigma)u_0(\tau)
\]
and the sequence of  functions \((b_j(\sigma,\tau))_{j\geq 1}\) can be constructed by recursion~\cite[Thm.~5.6]{BHR16}.
\end{subequations}

\subsection{Construction of quasi-modes}
Now we can introduce the following quasi-modes for \(\mathscr T_\hbar\) 
\begin{equation}\label{eq:N-t-phi}
  \tilde \phi_{\hbar,1}
  \coloneqq
  \chi\phi_{\hbar,1},
  \quad 
  \tilde\phi_{\hbar,2}
  \coloneqq
  \mathscr M_n\tilde\phi_{\hbar,1},\ldots\quad \tilde\phi_{\hbar,n}
  \coloneqq 
  \mathscr M_n\tilde\phi_{\hbar,n-1},
\end{equation}
where (see~\eqref{eq:def-Mgn-bnd})
\begin{equation}\label{eq:def-Mgn-bnd*}
  \mathscr M_nu(\sigma,\tau)
  =
  u\left(\sigma-\frac{L}{2n},\tau\right).
\end{equation}
Notice that \(\tilde\phi_{\hbar,i}\) is supported in \(\tilde\Gamma_{\hbar,\eta}^{(i)}\) defined as follows
\[ \tilde\Gamma_{\hbar,\eta}^{(i)}=\left( (i-2)\frac{2L}{n}+\eta,i\frac{2L}{n}-\eta\right)\times (0,\varepsilon_0\hbar^{-1/2}).\]
This yields quasi-modes for the operator \(\tilde T_h\) in~\eqref{eq:N-def-Th} obtained from \(\mathscr T_\hbar\) by the change of parameter \(\hbar=h^{1/2}\); more precisely,  we introduce
\begin{equation}\label{eq:N-t-u}
  \tilde u_{h,i}
  =
  \tilde\phi_{\hbar,i}\qquad (\hbar=h^{1/2}, ~ 1\leq i\leq n).
\end{equation}
Notice that by~\eqref{eq:N-WKB-a},
\[
  \ee^{\Phi_2(\sigma,\tau)}\bigl(\tilde u_{\hbar,2}(\sigma,t)- \ee^{2\ii \gamma_0 L/n\hbar}   \ee^{-\ii\gamma_0\sigma/\hbar}\tilde\psi_{\hbar,1}(\sigma-2L/n,\tau)\bigr)=\mathcal O(h^\infty),
\]
and by~\eqref{eq:gamma(s+s1)},
\[
  \begin{aligned}
    \Phi_2(\sigma)
    &=
    \Phi_1\left(\sigma-\frac{2L}{n}\right)\mbox{ on }\tilde\Gamma_{\hbar,\eta}^{(2)}\\
    \Phi_2(\sigma)+\Phi_1(\sigma)
    &=
    \sqrt{\frac{2C_1}{(\mu_1^N)''(\xi_0)}} \int_{\sigma-\frac{2L}{n}}^\sigma \sqrt{k_{\max}-k(\varsigma)}\dd \varsigma\\
    &=
    \mathsf S_n\mbox{ on }\tilde\Gamma_{\hbar,\eta}^{(2)}\cap \tilde\Gamma_{\hbar,\eta}^{(1)}
  \end{aligned}
\]
where \(\mathsf S_n\) is introduced in~\eqref{eq:def-S-N}.

To obtain quasi-modes for the operator \(T_h\) in~\eqref{eq:N-def-Th},   we truncate,  re-scale and pull back the quasi-modes \(\tilde u_{h,i}\) to the Cartesian coordinates via the transformation in~\eqref{eq:def-transf} and finally renormalize. More precisely, we introduce
\begin{equation}\label{eq:N-u}
  u_{h,i}(x)
  =
  h^{-1/4}\chi_0\big(t(x)/h^{1/2}\big) \tilde u_{h,i}\big(s(x),h^{-1/2}t(x)\big)
\end{equation}
where \(\chi_0\in C_c^\infty(\R;[0,1])\) satisfies \({\rm supp}\chi_0\subset(-\varepsilon_0,\varepsilon_0)\) and \(\chi_0=1\) on \([-\varepsilon_0/2,\varepsilon_0/2] \).

\subsection{Estimates of interaction coefficients}

We need to estimate
\[
  J_0(h)=\langle T_h u_{h,1},u_{h,1}\rangle\mbox{ and }J_1(h)=\langle T_hu_{h,1}, u_{h,2}\rangle.
\] 
By the exponential decay of \(u_{h,1}\) and \(u_{h,2}\),  we may write
\[
  \tilde J_0(h)
  =
  h\tilde J_0+\mathcal O(\ee^{-K'/h^{1/4}}),
  \quad 
  J_1(h)
  =
  h\tilde J_1+\mathcal O(\ee^{-K'/h^{1/4}})
\]
where \(K'>\mathsf S_n\) and
\[
  \tilde J_0
  =
  \langle \mathscr T_\hbar \tilde\phi_{\hbar,1},\tilde\phi_{\hbar,1}\rangle,\quad 
  \tilde J_1
  =
  \langle \mathscr T_\hbar \tilde\phi_{\hbar,1},\tilde\phi_{\hbar,2}\rangle  
\]
The term \(\tilde J_0\) is estimated as \( \mathcal O(\ee^{(-2\mathsf S_n+c\eta)/\hbar^{1/2}})\) where \(c\) is a positive constant independent of \(\hbar\) and \(\eta\). We fix now the choice of \(\eta\ll 1\) so that
\(2\mathsf S_n-c\eta>\mathsf S_n\).  

The term \(\tilde J_1\) is calculated as in~\cite[Sec.~7.2.1]{BHR}
\[
  \tilde J_1=  \ee^{2\ii (\gamma_0-\xi_0)L/n\hbar} \ee^{-\mathsf S_n/\hbar^{1/2}}\bigl(\hbar^{5/4} C_* +\mathcal O(\hbar^{7/4})\bigr)
\]
where \(C_*\in \C \setminus \{0\}\) is a constant independent of \(\hbar\). 

Writing \(C_*=|C_*|\ee^{\ii\alpha_0}\) and recalling the definition of \(\Phi_0\) in~\eqref{eq:def-flux-Om} and the relation 
\(\hbar=h^{1/2}\),  we get 
\begin{subequations}\label{eq:N-J}
\begin{equation}\label{eq:N-J1}
  J_1(h)\underset{h\searrow0}{\sim} |C_*|h^{13/8}\ee^{\ii\alpha_0  + 2\ii \Phi_0 /n h^{1/2}}\ee^{-\mathsf S_n/h^{1/4}}
\end{equation}
and 
\begin{equation}\label{eq:N-J0}
  J_0(h)\underset{h\searrow0}{=}o\big(J_1(h)\big).
\end{equation}
\end{subequations}

\subsection{Application when $n=3$.}

Let us assume that \(n=3\).   The functions  \(u_{h,i}\) introduced in~\eqref{eq:N-u} satisfy the conditions in Assumptions~\ref{ass1},   \ref{ass:symbis} (in Remark \ref{remlessabstract}) and \ref{ass-error}.   In  fact,  
\begin{itemize}
  \item
  By~\eqref{eq:N-t-phi} and~\eqref{eq:N-t-u},  we check that the symmetry invariance in Assumption~\ref{ass:symbis} is respected.
  \item The symmetry invariance and~\eqref{eq:N-WKB-a} ensure that \(u_{h,i}\) satisfy  the conditions in Assumption~\ref{ass1} with \(\mathfrak S_1,\mathfrak S_2,\mathfrak S_3\in(0,\mathsf S_n)\) arbitrarily close to \(\mathsf S_n\); we fix the choice so that \(\mathsf S_n<2\min_{1\leq j\leq 3}\mathfrak S_j\).
  \item 
  The estimates in~\eqref{eq:N-J} ensure that Assumption~\ref{ass-error} holds with \(\mathfrak S=\mathsf S_n\). 
\end{itemize}

Applying~\eqref{eq:formn=3},  we get a similar result to Theorem~\ref{thm:HK-3wells} (by following exactly the same argument).  In fact,   there is a relabeling  \(\mu_1(h),\mu_2(h),\mu_3(h)\) of the eigenvalues \(\lambda_1(h),\lambda_2(h),\lambda_3(h)\) of the Neumann magnetic Laplacian \(\mathcal L_h^N\)  with  the asymptotics 
\begin{equation}\label{eq:N-3wells}
  \begin{aligned}
    \mu_2(h)-\mu_1(h)&\underset{h\searrow0}{=} \bigl(\mathsf a(\Phi_0/3\sqrt{h}+\alpha_0)+o(1)\bigr) |C_*|\exp\left(\frac{- \mathsf S_3}{h^{1/2}}\right)\\
    \mu_3(h)-\mu_2(h) &\underset{h\searrow0}{=}\bigl(\mathsf b(\Phi_0/3\sqrt{h}+\alpha_0) +o(1)\bigr) |C_*|\exp\left(\frac{- \mathsf S_3}{h^{1/2}}\right)
  \end{aligned}
\end{equation}
where \(\mathsf a(\cdot)\) and \(\mathsf b(\cdot)\) are introduced~\eqref{eq:def-a-b}.
  
Moreover,  there exists a sequence \(\bigl((h_1(k),h_2(k),h_3(k)\bigr)_{k\geq k_0}\) which converges to \(0\) such that,  for all \(k\geq k_0\) we have
\[
  0<h_1(k+1)<h_3(k)<h_2(k)<h_1(k)<1
\]
and
\[
  \mu_1\big(h_1(k)\big)=\mu_3\big(h_1(k)\big),~  \mu_1\big(h_2(k)\big)
  =
  \mu_2\big(h_2(k)\big),~ \mu_2\big(h_3(k)\big)=\mu_3\big(h_3(k)\big).
\]
In particular,  this finishes the proof of Theorem~\ref{thm:HKS3}.

\section{Magnetic steps}\label{sec.edge}

\subsection{Introduction} 

In this section,  we work in the plane and the Hilbert space is \(H=L^2(\R^2)\).  Consider a positive integer \(n\) and denote by \(g=g_n\) the rotation in \(\R^2\) by \(2\pi/n\). We are then in the setting of Remark~\ref{remlessabstract} with the domain being all of \(\R^2\). Recall the definition of the transformation 
\[
  M(g)u(x)=u(g^{-1}x)\quad (u\in L^2(\R^2)).
\]

Let \(\Omega\subset\R^2\) be an open bounded subset of \(\R^2\) with \(C^\infty\) boundary \(\Gamma\). We assume that \(\Omega\) satisfies the conditions in \eqref{eq:4.1}.

Let \(\Ab:\R^2\to\R^2\) be a vector field such that
\begin{equation}\label{eq:def-ms}
  B\coloneqq \curl\Ab
  =
  \begin{cases}
    1&{\rm on~}\Omega\\
    \vartheta&{\rm on~}\R^2\setminus\overline{\Omega}
  \end{cases}
\end{equation}
where \(-1<\vartheta<0\) is a fixed constant. The magnetic field \(B=B_{\vartheta,\Omega}\) is a step function, hence called a \emph{magnetic step} (\cite{AHK} and references therein). The boundary \(\Gamma\) is the discontinuity curve of the magnetic step, and sometimes is called the \emph{magnetic edge} (\cite{FHK} and references therein).

Consider the Landau Hamiltonian on \(\R^2\)
\begin{equation}\label{eq:def-op-ms}
  \mathscr L_h^B
  =
  (-\ii h\nabla-\Ab)^2
\end{equation}
with semi-classical parameter \(h\) and magnetic field \(B\) as in \eqref{eq:def-ms}.  The symmetry conditions in \eqref{eq:4.1} allow us to prove that the eigenvalues of \(\mathscr L_h^B\) exhibit a braid structure in the semi-classical limit.  The proof and the construction are very similar to those for the Neumann problem in Section~\ref{sec.Neumann} modulo the following slight modifications:

\begin{enumerate}[i)]
  \item 
  Use the model operator on the full real line from \cite{AK20}.
  \item 
  Introduce adapted coordinates on a curved strip defined by the boundary \(\Gamma\),  via the signed normal distance to \(\Gamma\) and the tangential arc-length distance along \(\Gamma\).
  \item 
  Use a modified  single well operator,  with magnetic steps,  analyzed in \cite{AHK}.
  \item 
  Construct quasi-modes and apply the abstract results in Section~\ref{sec.IM}. 
\end{enumerate}

Most of the computations were carried out in \cite{AHK, FHK} so we will be rather succinct and just present the key constructions.

\subsection{Model on the real line}

On \(L^2(\R)\),  consider the family of Schr\"{o}dinger operators,   parameterized by \(\xi\in\R\),  
\begin{equation}\label{eq:ha}
  \mathfrak h_{\vartheta}[\xi]
  =
  -\frac{d^2}{d\tau^2}+V_{\vartheta}(\xi,\tau),
\end{equation}
where $\xi\in\mathbb R$ is a parameter and 
\begin{equation}\label{eq:potential}
  V_{\vartheta}(\xi,\tau)
  =
  \big(\xi+b_{\vartheta}(\tau)\tau\big)^2,\quad 
  b_{\vartheta}(\tau)
  =
  \mathbf{1}_{\mathbb R_+}(\tau)+\vartheta\mathbf{1}_{\mathbb R_-}(\tau)\,.
\end{equation}
We denote by  \(\mu_{\vartheta}(\xi)\) the ground state energy   of $\mathfrak h_{\vartheta}[\xi]$,  and according to \cite{AK20},  we can introduce the constants \(\beta_{\varsigma},\zeta_{\vartheta}\) as follows 
\begin{equation}\label{eq:beta}
  \beta_{\vartheta}
  \coloneqq 
  \inf_{\xi \in \mathbb R} \mu_{\vartheta}(\xi)=\mu_{\vartheta}(\zeta_{\vartheta})\,,
\end{equation}
where $\zeta_{\vartheta}<0$, is the unique minimum of $\mu_{\vartheta}(\cdot)$ and we have \(\mu_{\vartheta}''(\zeta_{\vartheta})>0\).    

Let $\phi_{\vartheta}$ be the positive $L^2$-normalized ground state of $\mathfrak  h_\vartheta[\zeta_{\vartheta}]$.  We introduce the negative constant~\cite{AK20}
\begin{equation}\label{eq:m3}
  M_3(\vartheta)
  =
  \frac 13\Big(\frac 1\vartheta-1\Big)\zeta_{\vartheta}\phi_{\vartheta}(0)\phi_{\vartheta}'(0).
\end{equation} 
\subsection{Adapted coordinates and single well problem}

Recall that \(\gamma:\R/2L\Z\to\R^2\) is a counter-clockwise oriented  arc-length parameterization of the curve \(\Gamma\) and that, for  \(s\in \R/2L\Z\),  \(\nu(s)\) is the unit normal to \(\Gamma\) at \(\gamma(s)\) which points to the interior of the \(\Gamma\). We can adjust the coordinates   \((s,t)\) introduced in \eqref{eq:def-transf},   by allowing \(t\) to have negative values.  We therefore set
\[
  \mathscr T(s,t)\coloneqq \gamma(s)-t\nu(s)\qquad ((s,t)\in \R/2L\Z\times \R).
\]
We introduce the curved strip 
\[
  \hat\Gamma(\varepsilon_0)=\{x\in\R^2~:~{\rm dist}(x,\Gamma)<\varepsilon_0\}
\]
and we choose \(\varepsilon_0>0\) so that \(\mathscr T\) is one-to-one on \(\R/2L\Z\times (-\varepsilon_0,\varepsilon_0)\) and that \(\hat\Gamma(\varepsilon_0)\) is the image of \(\R/2L\Z\times (-\varepsilon_0,\varepsilon_0)\) by \(\mathscr T\).

We assign to a function \(u\) defined on \(\hat\Gamma(\varepsilon_0)\) the function \(\tilde u=u\circ \mathscr T\) defined on \(\R/2L\Z\times (-\varepsilon_0,\varepsilon_0)\). The Hilbert space \(L^2(\hat\Gamma(\varepsilon_0))\) is then transformed to the weighted space  
\[
  L^2\big((\mathbb R/ 2L \mathbb Z) \times(-\varepsilon_0,\varepsilon_0); a\dd s\dd t\big),\quad  a(s,t)=1-tk(s)
\]
and the action of the transformation \(M(g)\) is still given by \eqref{eq:def-Mgn-bnd}.

Note that the action of $\mathscr L_h^B$ on \(\hat\Gamma(\varepsilon_0)\) is  transformed to  (after,  possibly,   a gauge transformation)
\begin{multline*}
  \tilde{\mathscr L}_h^B\coloneqq -h^2  a^{-1}\partial_t   a\partial_t
  + a^{-1}\left(-\ii h\partial_s+\gamma_0-b_{\vartheta}(t)t+\frac{k(s)}{2}b_{\vartheta}(t)t^2\right)\\
   \times a^{-1}\left(-\ii h\partial_s+\gamma_0-b_{\vartheta}(t)t+\frac{k(s)}{2}b_{\vartheta}(t)t^2\right)
\end{multline*}
where  the constant \(\gamma_0\) and the function \(b_{\vartheta}\) are  introduced in \eqref{eq:gamma0} and  \eqref{eq:potential} respectively.

We do the scaling \((\sigma,\tau)=(s,h^{-1/2}t)\) and introduce the effective semi-classical parameter \(\hbar=h^{1/2}\).  We obtain the operator
\begin{multline}\label{eq:tilde-Nh-bar-ms}
\mathscr N_{\hbar,\vartheta}=- a^{-1}_{\hbar}\partial_\tau   a_{\hbar}\partial_\tau
+  {a}^{-1}_{\hbar}\left(-\ii \hbar\partial_\sigma+{ \hbar^{-1}\gamma_0}-b_{\vartheta}\tau +\hbar \frac{k(\sigma)}{2}b_{\vartheta}\tau^2\right)\\
   \times {a}^{-1}_{\hbar}\left(-\ii \hbar\partial_\sigma+{ \hbar^{-1}\gamma_0}-b_{\vartheta}\tau +\hbar \frac{k(\sigma)}{2}b_{\vartheta}\tau^2\right)
\end{multline}
where\footnote{The function \(b_{\vartheta}\) now depends on the variable \(\tau\).}
\[
  a_\hbar(\sigma,\tau)=1-\hbar \tau k(\sigma).
\]
Our single well problem is defined as follows.  Fix a positive  \(\eta<\min(\frac14,\frac{L}{4n})\) and consider  
\[
  \hat\Gamma_{\hbar,\eta}^{(1)}=\Bigl(-\frac{2L}{n}+\eta,\frac{2L}{n}-\eta\Bigr)\times (-\varepsilon_0\hbar^{-1},\varepsilon_0\hbar^{-1}).
\]
We introduce the operator on \(L^2\big(\hat\Gamma_{\hbar,\eta}^{(1)}; a_{\hbar}\dd\sigma \dd\tau\big)\),
\begin{equation}\label{eq:tilde-Nh-bar-1-ms}
  \mathscr N_{\hbar,\vartheta}^{(1)}=\mathscr N_{\hbar,\vartheta}
\end{equation}
with domain 
\[
  \begin{aligned}
    \mathsf{Dom}(\mathscr N_{\hbar,\vartheta}^{(1)})=\{u\in L^2\big(\hat\Gamma_{\hbar,\eta}^{(1)};\dd\sigma \dd\tau\big)~:~&\partial_\tau^2u,\partial_\sigma^2u\in L^2\big(\hat\Gamma_{\hbar,\eta}^{(1)};\dd\sigma \dd\tau\big),\\
    &u|_{\tau=\pm\varepsilon_0h^{-1/2}}=0,~u|_{\sigma=\pm L/4n}=0\}.
  \end{aligned}
\]
We denote by \(\lambda(\hbar)\) the ground state energy of the operator \(\mathscr N_{\hbar,\vartheta}^{(1)}\); it is simple and can be expanded as follows~\cite{AHK, FHK}
\begin{equation}\label{eq:N-WKB-ev1-ms} 
  \lambda(\hbar)\underset{\hbar\to0}{\sim}\beta_{\varsigma}+M_3(\vartheta)k_{\max}\hbar+\sqrt{\frac{k_2M_3(\vartheta)\mu_{\vartheta}''(\zeta_{\vartheta})}{4} }\hbar^{3/2},
\end{equation}
and the splitting between \(\lambda(\hbar)\) and the second eigenvalue \(\lambda_2(\hbar)\)  is estimated as follows
\[
  \lambda_2(\hbar)-\lambda(\hbar)\underset{\hbar\to0}{\sim}\sqrt{\frac{k_2M_3(\vartheta)\mu_{\vartheta}''(\zeta_{\vartheta})}{4} }\hbar^{3/2}.
\]

\subsection{Quasi-modes and application}

In order to apply the results in Section~\ref{sec.IM} we introduce the operator
\[
  \mathcal T_h=\mathscr L_h^B-\ell(h)
\]
and construct suitable quasi-modes.  With \(\hbar=h^{1/2}\),  we choose \(\ell(h)=h\lambda(\hbar)\),  where \(\lambda(\hbar)\) is the ground state energy of the single well operator \(\mathscr N_{\hbar,\vartheta}^{(1)}\) introduced in \eqref{eq:tilde-Nh-bar-1-ms}.

Let \(\phi_{\hbar,1}^\vartheta\) be a normalized ground state of \(\mathscr N_{\hbar,\vartheta}^{(1)}\) and choose  \(\chi\in C^{\infty}(\R;[0,1])\) satisfying \eqref{eq:cond-chi-N}.   We introduce the quasi-modes
\begin{equation}\label{eq:N-t-phi-ms}
  \tilde \phi_{\hbar,1}^\vartheta\coloneqq \chi\phi_{\hbar,1}^\vartheta,\quad \tilde\phi_{\hbar,2}^\vartheta\coloneqq \mathscr M_n\tilde\phi_{\hbar,1}^\vartheta,\ldots\quad \tilde\phi_{\hbar,n}^\vartheta\coloneqq \mathscr M_n\tilde\phi_{\hbar,n-1}^\vartheta,
\end{equation}
where \(\mathscr M_n\) is introduced in \eqref{eq:def-Mgn-bnd*}.  We obtain quasi-modes for the operator \(T_h\) by  truncation,  re-scaling and pulling back to Cartesian coordinates; more precisely we introduce 
\begin{equation}\label{eq:N-u-ms}
  u_{h,i}^\vartheta(x)= h^{-1/4}\chi_0\big(t(x)/h^{1/2}\big) \tilde u_{h,i}^\vartheta\big(s(x),h^{-1/2}t(x)\big)
\end{equation}
where \(\chi_0\in C_c^\infty(\R;[0,1])\) is as in \eqref{eq:N-u} and \(\tilde u_{h,i}^\vartheta=\tilde\phi_{\hbar,i}^\vartheta\), with \(\hbar=h^{1/2}\).

We move then to the calculation of the following two terms
\[
  J_0(h)=\langle T_h u_{h,1},u_{h,1}\rangle\mbox{ and }J_1(h)=\langle T_hu_{h,1}, u_{h,2}\rangle.
\] 
That is essentially done in \cite{FHK}.  We introduce the following `distance'
\[
  \mathsf S_n(\vartheta)=\sqrt{\frac{ -2M_3(\vartheta)}{\mu_{\vartheta}''(\zeta_{\vartheta})}}\int_0^{2L/n}\sqrt{k_{\max} -k(s)}\dd s
\]
and the flux like term
\[
  \Phi_\vartheta=|\Omega|-2L\zeta_{\vartheta}.
\]
We have 
\[
  J_0(h)= \mathcal O(\ee^{(-2\mathsf S_n(\vartheta)+c\eta)/h^{1/4}})
\] 
where \(c\) is a positive constant independent of \(\hbar\) and \(\eta\),  and
\[
  \tilde J_1(h)=  \ee^{2\ii \Phi_\vartheta/n\hbar} \ee^{-\mathsf S_n(\vartheta)/h^{1/4}}\bigl( C_*(\vartheta)h^{13/8} +\mathcal O(h^{15/8})\bigr)
\]
where \(C_*(\vartheta)\in \C \setminus \{0\}\) is a constant independent of \(\hbar\). 

We fix now the choice of \(\eta\ll 1\) so that
\(2\mathsf S_n(\vartheta)-c\eta>\mathsf S_n(\vartheta)\) and we write \(C_*(\vartheta)=|C_*(\vartheta)|\ee^{\ii\alpha_0(\vartheta)}\).   We get 
\begin{subequations}\label{eq:N-J-ms}
\begin{equation}\label{eq:N-J1-ms}
  J_1(h)\underset{h\searrow0}{\sim} |C_*(\vartheta)|h^{13/8}\ee^{\ii\alpha_0 (\vartheta) + 2\ii \Phi_\vartheta/n h^{1/2}}\ee^{-\mathsf S_n(\vartheta)/h^{1/4}}
\end{equation}
and 
\begin{equation}\label{eq:N-J0-ms}
  J_0(h)\underset{h\searrow0}{=}o\big(J_1(h)\big).
\end{equation}
\end{subequations}

Let us now assume that \(n=3\),  which corresponds to the setting in Theorem~\ref{thm:HKS4}.     The functions  \(u_{h,i}^\vartheta\) introduced in~\eqref{eq:N-u} satisfy the conditions in Assumptions~\ref{ass1},   \ref{ass:symbis} (in Remark \ref{remlessabstract}) and \ref{ass-error}. Applying~\eqref{eq:formn=3},  we get   a relabeling  \(\mu_1(h),\mu_2(h),\mu_3(h)\) of the eigenvalues \(\lambda_1(h),\lambda_2(h),\lambda_3(h)\) of the Landau hamiltonian  \(\mathscr L_h^B\)  with  the asymptotics 
\[  
  \begin{aligned}
    \mu_2(h)-\mu_1(h)
    &
    \underset{h\searrow0}{=} \bigl(\mathsf a(\Phi_\vartheta/3\sqrt{h}+\alpha_0(\vartheta))+o(1)\bigr) |C_*(\vartheta)|\exp\left(\frac{- \mathsf S_3(\vartheta)}{h^{1/2}}\right)\\
    \mu_3(h)-\mu_2(h) 
    &
    \underset{h\searrow0}{=}\bigl(\mathsf b(\Phi_\vartheta/3\sqrt{h}+\alpha_0(\vartheta)) +o(1)\bigr) |C_*(\vartheta)|\exp\left(\frac{- \mathsf S_3(\vartheta)}{h^{1/2}}\right)
  \end{aligned}
\]
where \(\mathsf a(\cdot)\) and \(\mathsf b(\cdot)\) are introduced~\eqref{eq:def-a-b}.
  
Moreover, there exists a sequence \(\bigl((h_1(k),h_2(k),h_3(k)\bigr)_{k\geq k_0}\) which converges to \(0\) such that,  for all \(k\geq k_0\) we have
\[
  0<h_1(k+1)<h_3(k)<h_2(k)<h_1(k)<1
\]
and
\[
  \mu_1\big(h_1(k)\big)=\mu_3\big(h_1(k)\big),~  \mu_1\big(h_2(k)\big)
  =
  \mu_2\big(h_2(k)\big),~ \mu_2\big(h_3(k)\big)=\mu_3\big(h_3(k)\big).
\]
In particular, this finishes the proof of Theorem~\ref{thm:HKS4}.

\section*{Acknowledgments} 

The authors take the opportunity to thank the Knut and Alice foundation (grant KAW 2021.0259) for the possibility to host A.\ Kachmar in Lund for six months (1 Sep 2022 - 28 Feb 2023) and for supporting the visit of B. Helffer to Lund in January 2023.

\appendix

\section{Minimizing the function \(\Psi(r,t)\).}\label{sec:A}

Let \(\Psi\) be the function introduced in  \eqref{eq:def-Psi}.  Using the inequality \(\sqrt{t(t+1)}\leq t+\frac12\) we get
\begin{equation}\label{eq:A1}
  \Psi(r,t)\geq d(r)+\frac{|v_0^{\min}|}{2}\ln\Bigl(1+\frac1t\Bigr)+\frac{(L-r)^2}{2}\left(t+\frac12\right).
\end{equation}
Consequently,  the minimum of \(\Psi(r,t)\) over \([a,L-a]\times\overline\R_+\) is attained at \((r_0,t_0)\in [a,L-a]\times\R_+\). 
By \cite[Remark~6.8]{HK22},
\begin{equation}\label{eq:A2}
  \inf_{(r,t)\in[a,L-a]\times\R_+}{\Psi (r,t)} \geq \inf\left\{\inf_{(r,t)\in [0,a]\times\R_+}\Psi(r,t),\inf_{t\in\R_+}\Psi(L-a,t) \right\}.
\end{equation}
Recalling \(F(v_0)\) from \eqref{eq:def-Psi},   we have by \cite[Eq.~(6.18) and Prop.~6.5]{HK22},  %
\begin{equation}\label{eq:A3}
  -F(v_0)+\inf_{(r,t)\in [0,a]\times\R_+}\Psi(r,t)= S(v_0,L)\geq S_a
\end{equation}
where  \(S(v_0,L)\mbox{ and }S_a\) are introduced in \eqref{eq:defSs}.

The function 
\[
  G(t)
  := 
  \frac{|v_0^{\min}|}{2}\ln\Bigl(1+\frac1t\Bigr)+\frac{a^2}{2}\left(t+\frac12\right)
\]
has a unique minimum on \(\R_+\),
\[
  t_*
  =
  \sqrt{\frac14+\frac{|v_0^{\min}|}{a^2}}-\frac12.
\]
By \eqref{eq:A1} (for \(r=L-a\)),  we get
\[
  \begin{aligned}
    \inf_{t\in\R_+}\Psi(L-a,t)
    &
    \geq d(L-a)+G(t_*)\\
    &
    =
    d(L-a)+\frac{a}{4}\sqrt{a^2+4|v_0^{\min}|}+\frac{|v_0^{\min}|}{2}\ln \frac{ \left(\sqrt{a^2+4|v_0^{\min}|} +a\right)^2}{4|v_0^{\min}|}.
  \end{aligned}
\]
Consequently,
\begin{equation}\label{eq:A4}
  -F(v_0)+\inf_{t\in\R_+}\Psi(L-a,t)\geq d(L-a)+d(a)=S_a.
\end{equation}
Collecting \eqref{eq:A3} and \eqref{eq:A4}, we infer from \eqref{eq:A2} that
\[
  -F(v_0)+\inf_{(r,t)\in[a,L-a]\times\R_+}\Psi(r,t)\geq S_a.
\]
\subsection*{Conflict of interest} The authors have no conflict of interest to declare.


\begin{thebibliography}{99}

\bibitem{Al} K. Abou Alfa.
\newblock Tunneling effect in two dimensions with vanishing magnetic fields.
\newblock Preprint arXiv:2212.04289 (2022).

\bibitem{AK20}
W.~Assaad and A.~Kachmar.
\newblock Lowest energy band function for magnetic steps.
\newblock {\em J. Spectr. Theory}
\newblock
 12(2):  813--833 (2022).
	
\bibitem{AHK} W. Assaad, B. Helffer, and A. Kachmar.   Semi-classical eigenvalue estimates under magnetic steps.  arXiv:2108.03964 (2021). To appear in Analysis and PDE.

\bibitem{BDMV} V. Bonnaillie-No$\rm\ddot{e}$l,  M. Dauge,  D. Martin,  G.  Vial.  Numerical computations
of fundamental eigenstates for the Schr$\rm\ddot{o}$dinger operator under
constant magnetic field.   {\it Comput.  Methods Appl.  Mech.   Engng.} 196:
3841--3858 (2007).

\bibitem{BHR} V. Bonnaillie-No\"{e}l, F. H\'erau, and N. Raymond. Purely magnetic tunneling effect in two dimensions. {\em Invent.  Math.}  227(2): 745--793 (2022).

\bibitem{BHR16} V. Bonnaillie-No\"{e}l, F. H\'erau, and N. Raymond.    \newblock
Magnetic WKB constructions.  \newblock Arch. Ration. Mech. Anal.  221(2):   817--891 (2016).

\bibitem{FSW} C.  Fefferman, J. Shapiro, M. Weinstein.
Lower bound on quantum tunneling for strong magnetic fields. {\it 
SIAM J. Math. Anal.} 54(1), 1105--1130 (2022). (see also arXiv:2006.08025v3).

\bibitem{FH-b} S. Fournais and B. Helffer. Spectral methods in surface superconductivity. Progress in Nonlinear Differential Equations and Their Applications 77. Basel: Birkh\"auser (2010).

\bibitem{FH06}  S. Fournais and B. Helffer.  \newblock Accurate eigenvalue asymptotics for the magnetic Neumann Laplacian.\newblock Ann.  Inst.   Fourier 56(1):  1--67 (2006).

\bibitem{FHK} S. Fournais, B. Helffer, A. Kachmar.
\newblock Tunneling effect induced by a curved magnetic edge. 
\newblock
R.L. Frank (ed.) et al., The physics and mathematics of Elliott Lieb. The 90-th anniversary. Volume I. Berlin: European Mathematical Society (EMS).   315--350 (2022).

\bibitem{FHKR} S.  Fournais, B.  Helffer,  A. Kachmar,  N. Raymond.  Effective operators on an attractive magnetic edge. 
Journal de l'\'Ecole polytechnique — Math\'ematiques,  Vol.  10: 917--944 (2023).

\bibitem{Ha} E. Harrell.
\newblock Double wells.
\newblock Comm. Math. Phys. 75: 239--261 (1980).

\bibitem{He88} B. Helffer. Semi-classical analysis for the Schr\"odinger operator and applications. Lecture Notes in Mathematics, Vol.~1336, Berlin : Springer-Verlag  (1988).

\bibitem{HK22} B. Helffer, A. Kachmar. Quantum tunneling in  deep potential wells and  strong magnetic field revisited.  ArXiv preprint: arXiv:2208.13030v4 (2022).

\bibitem{HeSj1} B. Helffer and J. Sj$\ddot{\rm o}$strand.
\newblock Multiple wells in the semi-classical limit I.
\newblock Comm. in PDE, vol.9: 337--408 (1984).

\bibitem{HeSj2} B. Helffer and J. Sj$\ddot{\rm o}$strand.
\newblock Puits multiples en limite semi-classique. II Interaction mol\'eculaire. Sym\'etries. Perturbation.
\newblock Ann. IHP, Section A 42(2): 127--212 (1985). 

\bibitem{HeSjPise} B. Helffer and J. Sj$\ddot{\rm o}$strand.   Effet tunnel pour l'\'equation de Schr$\ddot{\rm o}$dinger avec champ magn\'etique. \newblock
Ann. Sc. Norm.  Super. Pisa,  Cl. Sci., IV. Ser. 14(4),   625--657 (1987).

\bibitem{HHHO} B. Helffer, M.  Hoffmann-Ostenhof,  T.  Hoffmann-Ostenhof,  M.P.  Owen.
Nodal sets for ground states of Schr\"odinger operators with zero magnetic field in non simply connected domains. 
Commun.  Math.  Phys.  202(3): 629--649 (1999).

\bibitem{KR} A. Kachmar,  N. Raymond.
\newblock   Tunnel effect in a shrinking shell
enlacing a magnetic field. 
\newblock  {\it Rev.  Mat.  Iberoam.} 35(7):  2053--2070 (2019).

\bibitem{Si1} B. Simon.
\newblock  Semiclassical analysis of low lying eigenvalues, I. Non-degenerate minima: Asymptotic expansions, \newblock Ann. Inst. H. Poincar\'e 38:   295--307 (1983).

\bibitem{Si2} B. Simon.
\newblock Semiclassical analysis of low lying eigenvalues, II. Tunneling, 
\newblock Ann. of Math. 120:  89--118 (1984).
\end{thebibliography}
\end{document}